\documentclass[12pt]{amsart}

\usepackage{amssymb}
\usepackage{url}
\usepackage{enumitem}

\makeatletter
\@namedef{subjclassname@2020}{%
  \textup{2020} Mathematics Subject Classification}
\makeatother

\newtheorem{theorem}{Theorem}[section]

\newtheorem{lemma}[theorem]{Lemma}
\newtheorem{proposition}[theorem]{Proposition}

\theoremstyle{definition}

\numberwithin{equation}{section}

\DeclareMathOperator{\disc}{disc} % mine
\DeclareMathOperator{\gal}{Gal} % mine
\DeclareMathOperator{\rk}{rk} % mine

\newcommand{\Z}{{\mathbb Z}} % mine
\newcommand{\N}{{\mathbb N}} % mine
\newcommand{\Q}{{\mathbb Q}} % mine
\newcommand{\Ell}{{\mathcal L}} % mine
\renewcommand{\o}{{\mathcal O}} % mine

\begin{document}

\title[Ideal Class Searches]{Relative Ideal Classes\\ of Arbitrary Order}

\author[D. L. Pincus]{David L. Pincus}
\address{Department of Mathematics\\ University of Maryland\\
College Park, MD 20742, USA}
\email{dlpincus@gmail.com}

\author[L. C. Washington]{Lawrence C. Washington}
\address{Department of Mathematics\\ University of Maryland\\
College Park, MD 20742, USA}
\email{lcw@umd.edu}

\date{}

\begin{abstract}
%A template for articles in IMPAN journals in the \texttt{amsart} style. Using \texttt{pdflatex} is strongly preferred.
We adapt a known technique for searching for ideal classes of arbitrary order and then apply it to 
three families of number fields.  We show that a family of cyclic sextic number fields has infinitely many fields in it that contain a relative ideal class of order $r,$ where $r$ is a positive integer relatively prime to the degree of the extension.  We then show that the same holds true for a family of cyclic quartic number fields.  
Though the technique is traditionally applied to Galois extensions, we show how it may be adapted to handle a family of non-Galois cubic number fields and prove that 
this family contains infinitely many fields with an ideal class of arbitrary order relatively prime to three.
\end{abstract}



\maketitle

\section{Introduction}

Let $\frak{F}:=\{K_n\}$ denote an indexed family of number fields and let $r$ be a positive integer.  Let $\mathcal{I}_{K_n}$ denote the group 
of fractional ideals of $K_n \in \frak{F}.$
Can we find a field $K_n \in \frak{F}$ that contains an ideal class of order 
$r?$  We may answer in the affirmative if the following criteria are satisfied.\\
\begin{enumerate}\label{intro_list}
\item We can find an index $n$ and a principal ideal $w\o_{K_n} \in \mathcal{I}_{K_n}$ that factors 
as $w\o_{K_n} = J^r$ for some fractional ideal $J \in \mathcal{I}_{K_n}.$\label{intro_lista}
\item We can show that no smaller power of $J$ is principal.\label{intro_listb}\\
\end{enumerate}
This sort of approach has been used by many authors going back at least as far as Nagell \cite{nagell}.\\  
\\
A well-known method for meeting Criterion (\ref{intro_lista}) 
begins by looking for solutions to the equation 
\begin{equation}\label{intro_eqn1}
Y^r = f_n(X), 
\end{equation}
where $f_n(X) \in \Q\lbrack X \rbrack$ is the minimal polynomial of a primitive element for $K_n/\Q.$
Indeed if $K_n/\Q$ is Galois with primitive element $\rho,$ then a solution 
$(x,y) \in \Q^2$ to Equation (\ref{intro_eqn1}) gives 
\begin{equation}\label{intro_eqn2}
y^r=f_n(x) = \prod_{i=0}^{\lbrack K : \Q \rbrack - 1}(x-\rho_i),
\end{equation}
where the $\rho_i$ run over the Galois conjugates of $\rho:=\rho_0.$  If one can show that 
the exact power of a prime ideal that divides, say, $(x-\rho_i)\o_K$ is also the exact power of the same 
prime ideal that divides $y^r\o_K,$ then one may conclude that $(x-\rho_i)\o_{K_n}=J^r$ 
for some ideal $J \in \mathcal{I}_{K_n}.$\\  
\\
A well-known method for meeting Criterion (\ref{intro_listb}) is then to assume 
the contrary and suppose that $(x-\rho_i)\o_{K_n}$ is the $p$th power of a principal ideal, where 
$p$ is a prime dividing $r.$  Element-wise this implies that there is a $z \in K_n^\times$ 
and a unit $u \in \o_{K_n}^\times$ such that
\begin{equation}\label{intro_eqn3}
x-\rho_i = uz^p.
\end{equation}
By restricting our consideration to those solutions $(x,y) \in \Q^2$ that have the property that 
for each prime divisor $p$ of $r$ the prime divisors of $y$ 
satisfy certain $p$th power residue conditions, we can ensure that for each prime 
$p \mid r,$ Equation (\ref{intro_eqn3}) fails to hold.\\
\\
In this paper we extend the technique outlined above in the following three ways.\\
\\
\noindent {\bf A.} If several of the (ideal) factors $(x-\rho_i)\o_{K_n}$ on the right hand side of Equation (\ref{intro_eqn2}) factor as 
$r$th powers of ideals, then so do products and quotients of these ideals.  By taking appropriate 
products and quotients of such factors, we may obtain a principal ideal 
that has trivial norm down to each proper subfield of $K_n$ and which also 
factors as an $r$th power, say as $J^r.$  The ideal class of $J$ 
then has trivial norm down to each proper subfield of $K_n$ and is 
called a \emph{relative ideal class}.  We use this technique in Section 2 (resp. Section 3) below to 
search for relative ideal classes of order $r$ in a family of sextic (resp. quartic) number fields.  The qualifier 
\emph{relative} is important.  Had we simply been interested in searching for ideal classes of order $r,$ then 
it would have sufficed to look for a principal fractional ideal belonging to any of the proper subfields of $K_n$ that factors as an $r$th power.  The fact that the ideal classes we detect are relative ideal classes guarantees that these classes are \emph{new} and not coming from ideal classes of any proper subfield.\\
\\
\noindent {\bf B.}  If $K_n/\Q$ is non-Galois with primitive element $\rho \in \o_{K_n}$ and $(x,y)\in \Z^2$ is such that 
$y^r = f_n(x),$ then Equation (\ref{intro_eqn2}) no longer holds in $K_n.$  
It is still true, however, that 
$y^r = f_n(x)=(x-\rho)q(x),$
where $q(x)\in \o_{K_n}.$  In particular, if for each prime ideal 
$\frak{P}$ dividing $(x-\rho)\o_{K_n},$ we have $q(x) \not\equiv 0 \pmod{\frak{P}},$ then 
we still have that $(x-\rho)\o_{K_n} = J^r$ for some ideal $J \in \mathcal{I}_{K_n}.$  We use 
this technique in Section 4 to search for ideal classes of order $r$ in a family of non-Galois 
cubic number fields.\\
\\
\noindent {\bf C.}  Congruence conditions are not always sufficient to guarantee that 
for each prime $p$ dividing $r,$ 
$(x-\rho_i)/u$ doesn't have a $p$th root in $K_n.$  Likewise, congruence 
conditions don't always suffice to guarantee that for each prime $p \mid r,$ 
there is no $p$th root of $w/u$ in $K_n,$
where $w$ is a product of powers of the $(x-\rho_i).$  What the congruences can do, however, is limit the 
possibilities for $u \in \o_{K_n}^\times.$  Eliminating the possibility that there is 
a prime $p \mid r$ such that $w/u$ has a $p$th root in $K_n$ can then be reduced to 
showing that the polynomial $\varphi(X^r)$ is irreducible in $\Q\lbrack X \rbrack,$ where 
$\varphi(X) \in \Q\lbrack X \rbrack$ is the minimal polynomial of $w/u.$  Since $\varphi(X^r)$ depends on both $r$ and the parameter $n$ indexing our families of number fields, proving irreducibility is not easy.  In Sections 2 and 3 we make use of Hilbert's Irreducibility Theorem \cite{hilbert} to do so.\\  
\\
Using these techniques we are able to show that each of the families mentioned above has 
infinitely many fields that contain an ideal class of order $r$.  In the case of 
the sextic and quartic families, these are \emph{relative} ideal classes.  
In all cases $r$ is a positive integer relatively 
prime to the degree of the extension, but otherwise arbitrary.  Although the methods we employ may be extended to handle the case when $(r,\lbrack K : \Q \rbrack)>1,$ the result clutters the 
presentation and obscures the techniques we are trying to illustrate.  Thus our presentation sticks to the 
case in which $(r,\lbrack K : \Q \rbrack)=1.$

% Sextic Extensions
\section{A family of sextic number fields.}

Let $n \in \Z \smallsetminus \{0,\pm 6,\pm 26\}$ and let $K_n = \Q(\rho),$ where $\rho$ is a root of 
\begin{equation}\label{deg6_eqn1}
\begin{split}
f_n(X) := &X^6 -\left(\frac{n-6}{2}\right)X^5-5\left(\frac{n+6}{4}\right)X^4\\
&-20X^3+5\left(\frac{n-6}{4}\right)X^2+\left(\frac{n+6}{2}\right)X+1.
\end{split}
\end{equation}
$K_n/\Q$ is a real cyclic sextic extension and was studied extensively by M.-N. Gras in \cite{gras1}.
Let $G=\gal(K_n/\Q).$  Then $G = \langle \sigma \mid \sigma^6 = 1 \rangle,$ where 
$\sigma: \rho \mapsto (\rho-1)/(\rho+2).$  The images of 
$\rho$ under $G,$ denoted
$\rho_i := \sigma^i(\rho),$ $(i=0,\ldots,5),$ are given by
\begin{equation}\label{deg6_eqn4}
\begin{gathered}
\rho_0 = \rho\\
\rho_3 = -(\rho+2)/(2\rho+1)
\end{gathered} \qquad
\begin{gathered}
\rho_1 = (\rho-1)/(\rho+2)\\
\rho_4 = -(\rho+1)/\rho
\end{gathered} \qquad
\begin{gathered}
\rho_2 = -1/(\rho+1)\\
\rho_5 = -(2\rho+1)/(\rho-1).
\end{gathered}
\end{equation}

\noindent Let $\frak{F}:= \{K_n \mid n \neq 0, \pm 6, \pm 26\}.$  Let $\mathcal{I}_K$ denote the group 
of fractional ideals of $K \in \frak{F}$ and for a fractional ideal $J \in \mathcal{I}_K,$ let
$\lbrack J \rbrack$ denote its ideal class.  Let $k_2 = \Q(\sqrt{n^2+108})$ denote the unique 
quadratic subfield of $K$ and let $k_3 = \Q(\rho_0\rho_3)$ denote the unique cubic subfield 
of $K.$   Then $\lbrack J \rbrack$ is said to be a \emph{relative ideal class} if it has trivial norm down to both 
$k_2$ and $k_3.$\\  
\\
\noindent Note that $\rho_i,$ $(i=0,\ldots,5)$ is an algebraic integer if and only if 
$n \equiv 2 \pmod{4}.$  So it is also worthwhile to note that $K_n = \Q(\rho_0/\rho_3)$ and that $\rho_0/\rho_3$ is a zero 
of the polynomial
\begin{align}\label{deg6_eqn2}
\begin{split}
\varphi_n(X) &:= (X-1)^6-(n^2+108)(X^2+X)^2\\
&= X^6 -6X^5-(n^2+93)X^4-(2n^2+236)X^3-(n^2+93)X^2-6X+1.
\end{split}
\end{align}
In particular, $\rho_0/\rho_3$ and its Galois conjugates are algebraic integers for any $n,$ i.e., regardless of whether or 
not $n \equiv 2 \pmod{4}.$  Since 
$N_{K \mid k_2}(\rho_0/\rho_3) = 1$ and $N_{K \mid k_3}(\rho_0/\rho_3) = 1$ the Galois conjugates 
of $\rho_0/\rho_3$ are, in fact, relative units of $K_n.$\\
\\
\noindent Our primary result is 

\begin{theorem}\label{deg6_thm1}
Let $r>0$ be an integer with $(6,r)=1.$  There are infinitely many fields $K \in \frak{F}$ containing a 
relative ideal class of order $r.$
\end{theorem}

\noindent It is clear that $K_n = K_{-n},$ since the zeros of $\varphi_n(X)$ serve as primitive elements 
for $K_n/\Q$ and since $\varphi_n(X) = \varphi_{-n}(X).$  
As we make no additional assumptions regarding the injectivity of the parametrization of $\frak{F}$ given 
by $n \mapsto K_n,$ it would be insufficient for us to show that there are infinitely many $n$ for which  
$K_n$ contains a relative ideal class of order $r.$  Instead we put most of our effort into first proving 

\begin{proposition}\label{deg6_prop1}
Let $r>0$ be an integer with $(6,r)=1$ and let $q \in \Z$ be a prime such that
$$
\mathcal{P}(X) := (X^r+143-30\sqrt{-108})(X^r+143+30\sqrt{-108}) \in \Z\lbrack X \rbrack
$$
splits into linear factors $\pmod{q}.$  Assume that $q \nmid 210r.$  There is a field $K \in \frak{F}$ that satisfies the following three conditions:
\begin{enumerate}[label=(\roman*)]
\item There is an ideal $J \in \mathcal{I}_K$ whose ideal class, $\lbrack J \rbrack,$ has order $r.$
\item $N_{K \vert k_i}(J),$ $(i=2,3)$ is a principal ideal.
\item $q$ ramifies in $k_2.$
\end{enumerate}
\end{proposition}

\noindent Once Proposition \ref{deg6_prop1} is established we have the following. 

\begin{proof}[Proof of Theorem \ref{deg6_thm1}]  There are infinitely many primes $q \nmid 210r$ for which 
$\mathcal{P}(X)$ splits into linear factors $\pmod{q}.$  Hence there are infinitely many primes $q$ 
that ramify in the quadratic subfield of some $K \in \frak{F}$ (depending on $q$) for which 
$\mathcal{I}_K$ contains a relative ideal class of order $r.$  Since only finitely many primes can ramify in 
the quadratic subfields of a finite subcollection of $\frak{F},$ we conclude that the collection of $K \in \frak{F}$ 
for which $\mathcal{I}_K$ contains a relative ideal class of order $r$ is infinite.
\end{proof}

\noindent We remark that we could have proved Theorem \ref{deg6_thm1} using the argument given in 
the proof of Theorem \ref{ng_thm1} below.  That proof finds infinitely many fields containing a class of 
order $r$ by finding cyclic subgroups of the class group of orders that are increasingly higher powers of $r.$  Since the proof we've given above doesn't rely on such an argument, it reveals slightly more about this family of fields.\\

\noindent To provide some insight into the motivation behind Proposition \ref{deg6_prop1}, we note that 
when $\mathcal{P}(X)$ splits into linear factors $\pmod{q},$ there is a root of 
the equation 
\begin{equation}\label{deg6_eqn2}
X^r=30n-143=f_n(-3)
\end{equation}
in the field $\Z/q\Z,$ where $n^2+108 \equiv 0 \pmod{q}.$\\
\\
\noindent Let 
$$
w := \frac{(3+\rho_1)(3+\rho_2)}{(3+\rho_4)(3+\rho_5)}.
$$

\begin{proposition}\label{deg6_prop2}
Let $n \equiv 2 \pmod{4}$ be an integer.  If there is a $y \in \Z$ with $7\nmid y$ and such that $y^r = 30n-143=f_n(-3),$ then $w\o_{K_n} = J^r$ for some ideal $J \in \mathcal{I}_{K_n}.$
\end{proposition}

\noindent For the proof of this proposition we require the following easy lemma:

\begin{lemma} \label{deg6_lem1}
If a prime $p$ is a common divisor of $30n-143$ and $\disc(f_n),$ then $p = 7.$
\end{lemma}

\begin{proof}
The congruences $30n-143 \equiv 0 \pmod{p}$ and $\disc(f_n)= 3^6(n^2+108)^5/2^{14}\equiv 0 \pmod{p}$ hold simultaneously only if $p=7.$
\end{proof}

\begin{proof}[Proof of Proposition \ref{deg6_prop2}]
It suffices to prove that for each index $i,$ the ideal $(3+\rho_i)\o_K$ is an $r$th power.
Since $n \equiv 2 \pmod{4},$ each $\rho_i$ is an algebraic integer of $K.$  
Since $7 \nmid y,$ no prime dividing $y$ also divides $\disc(f_n).$  If a prime 
$p \mid y$ and $\frak{P}\mid p\o_K$ is a prime ideal then $\prod_{i=0}^5(3+\rho_i) = f_n(-3) \in \frak{P}.$  
Hence $3+\rho_i \in \frak{P}$ for some $0 \le i \le 5.$  
Each of the six Galois conjugates of $3+\rho_i$ is 
contained in a Galois conjugate of $\frak{P},$ but since $p \nmid \disc(f_n),$ no two distinct conjugates of 
$3+\rho_i$ can be contained in the same Galois conjugate of $\frak{P}.$
Hence the prime ideals $\frak{P}^{\sigma^i},$ $(i=0,\ldots,5),$ are all distinct and 
each contains exactly one of the Galois conjugates of $3+\rho_i.$
So for each prime $p \mid y,$ let $\frak{P}_i$ denote the unique prime common divisor of 
$p\o_K$ and $(3+\rho_i)\o_K,$ $(i=0,\ldots,5).$
If $p^{\nu(p)}$ is the exact power of $p$ dividing $y,$ then $\frak{P}_i^{r\nu(p)}$ is the exact power 
of $\frak{P}_i$ dividing $(3+\rho_i)\o_K.$  Hence 
$(3+\rho_i)\o_K = \left(\prod_{p \mid y} \frak{P}_i^{\nu(p)}\right)^r.$
\end{proof}

\noindent Let $\varepsilon$ denote the fundamental unit of $k_2,$ let $\mu_0,\mu_1$ be a fundamental 
system of units for $k_3$ and let $U:=\langle -1, \varepsilon, \mu_0, \mu_1, \rho_0/\rho_3, \rho_1/\rho_4 \rangle \subseteq \o_K^\times.$  Without loss of generality \cite{gras3} assume that 
$\mu_i = \sigma^{i}(\mu_0),$ $(i=0,1,2).$  We will need to know that $\lbrack \o_K^\times : U \rbrack < \infty.$
Since $\rk \o_K^\times = 5,$ it will suffice to prove the following.

\begin{lemma}
The subset $S:=\{\varepsilon, \mu_0, \mu_1, \rho_0/\rho_3, \rho_1/\rho_4\} \subset \o_K^\times$ is multiplicatively independent.
\end{lemma}

\begin{proof}
Let $R$ denote the regulator of $S.$  Then $R$ is the absolute value of
\begin{align*}
\begin{split}
\label{eqn3f}
&\;\;\;\;\;\det \begin{pmatrix}
\log \vert \mu_0 \vert &\log \vert \mu_1 \vert &\log \vert \rho_0/\rho_3 \vert & \log \vert \rho_1/\rho_4 \vert & \log \vert \epsilon \vert\\
\log \vert \mu_1 \vert &\log \vert \mu_2 \vert &\log \vert \rho_1/\rho_4 \vert & \log \vert \rho_2/\rho_5 \vert & \log \vert \epsilon^\sigma \vert\\
\log \vert \mu_2 \vert &\log \vert \mu_0 \vert &\log \vert \rho_2/\rho_5 \vert & \log \vert \rho_3/\rho_0 \vert & \log \vert \epsilon \vert\\
\log \vert \mu_0 \vert &\log \vert \mu_1 \vert &\log \vert \rho_3/\rho_0 \vert & \log \vert \rho_4/\rho_1 \vert & \log \vert \epsilon^\sigma \vert\\
\log \vert \mu_1 \vert &\log \vert \mu_2 \vert &\log \vert \rho_4/\rho_1 \vert & \log \vert \rho_5/\rho_2 \vert & \log \vert \epsilon \vert\\
\end{pmatrix}\\
\\
&= \det
\begin{pmatrix}
\log \vert \mu_0 \vert &\log \vert \mu_1 \vert &\log \vert \rho_0/\rho_3 \vert & \log \vert \rho_1/\rho_4 \vert & \log \vert \epsilon \vert\\
\log \vert \mu_1 \vert &\log \vert \mu_2 \vert &\log \vert \rho_1/\rho_4 \vert & \log \vert \rho_2/\rho_5 \vert & \log \vert \epsilon^\sigma \vert\\
0 &0 &2\log \vert \rho_1/\rho_4 \vert &2\log \vert \rho_2/\rho_5 \vert & \log \vert \epsilon \vert\\
0 &0 &2\log \vert \rho_2/\rho_5 \vert &2\log \vert \rho_3/\rho_0 \vert & \log \vert \epsilon^\sigma \vert\\
0 &0  &0 &0 & 3\log \vert \epsilon \vert
\end{pmatrix}.
\end{split}
\end{align*}
It is clear that 
$$
\det
\begin{pmatrix}
\log \vert \mu_0 \vert &\log \vert \mu_1 \vert\\
\log \vert \mu_1 \vert &\log \vert \mu_2 \vert
\end{pmatrix}
\neq 0,
$$
because up to sign this is the regulator of $\{\mu_0,\mu_1\},$ a multiplicatively independent 
subset of $k_3.$  It is equally clear that 
$3\log \vert \varepsilon \vert \neq 0,$ because $\varepsilon > 1.$\\
\\
When $n \ge 76,$ three of the zeros of $\varphi_n(X)$ lie in the intervals $(3-n,4-n),$ $(-1,-9/10),$ and $(1/(n+5),1/(n+4))$;
the other three zeros of $\varphi_n(X)$ are reciprocals of these zeros.  Since $f_n(1)=-27$ and since 
$f_n(x) \to \infty$ as $x \to \infty,$ we deduce that $f_n(X)$ has a zero $>1.$  From 
the form of this zero's Galois conjugates, it is clear that $f_n(X)$ has another positive zero in the interval 
$(0,1),$ and that the remaining zeros of $f_n(X)$ are all negative.  If we take $\rho_0$ to be the zero of 
$f_n(X)$ greater than $1,$ then 
we may identify $\rho_0/\rho_3$ with the zero of $\varphi_n(X)$ in the interval $(3-n,4-n).$  
We may also then identify $\rho_1/\rho_4$ with the zero in $(-1,-9/10)$ and $\rho_2/\rho_5$ with 
the zero in $(1/(n+5),1/(n+4)).$  Hence 
\begin{align*}
\log\vert\rho_1/\rho_4\vert\log\vert\rho_3/\rho_0\vert-\log^2\vert\rho_2/\rho_5\vert 
&< -\log(9/10)\log(n-3)-\log^2(n+4)\\
&<-\log(n+4)(\log(9/10)+\log(n+4))\\
&<0,
\end{align*}
and we conclude that $R \neq 0$ when $n \ge 76.$  A computer verifies 
that $\log\vert\rho_1/\rho_4\vert\log\vert\rho_3/\rho_0\vert-\log^2\vert\rho_2/\rho_5\vert  \neq 0$ when $0< n < 76$ and $n \not\in \{0,6,26\}.$  Hence $R \neq 0$ and $S$ is a multiplicatively independent subset of units of $K.$
\end{proof}

\noindent Having shown that the index of $U$ in $\o_K^\times$ is finite, we will 
use the following lemma to prevent certain primes from dividing this index.  We also note that 
in the sequel we employ the notation 
$(a\,\vert\,\ell)_p=1$ (resp. $(a\,\vert\,\ell)_p\neq 1$) to mean that 
$a$ is a $p$th power $\textmd{residue}\pmod{\ell}$ (resp. $a$ is a $p$th power $\textmd{nonresidue}\pmod{\ell}.$)  
When the subscript $p=2,$ the notation is the Legendre symbol.

\begin{lemma}
Let $n \equiv 2 \pmod{4}$ be an integer and let $p\neq 2,3$ be a prime.
If there are primes $\ell_i \mid 30n-143=f_n(-3),$ $(i=1,2),$ not equal to $7$ and such that
\begin{enumerate}
\item $(2\,\vert\,\ell_1)_p = (3\,\vert\,\ell_1)_p = 1$ and $(5\,\vert\,\ell_1)_p\neq 1$
\item $(3\,\vert\,\ell_2)_p = (5\,\vert\,\ell_2)_p = 1$ and $(2\,\vert\,\ell_2)_p\neq 1,$ 
\end{enumerate}
then $p \nmid \lbrack \o_K^\times : U \rbrack.$
\end{lemma}

\begin{proof}
If $p \mid \lbrack \o_K^\times : U \rbrack,$ then there is a $u \in \o_K^\times \smallsetminus U$ such that 
$u^p = \pm \varepsilon^a\mu_0^b\mu_1^c(\rho_0/\rho_3)^d(\rho_1/\rho_4)^e.$  Without loss of generality, 
we may take $u^p = \varepsilon^a\mu_0^b\mu_1^c(\rho_0/\rho_3)^d(\rho_1/\rho_4)^e$ with 
$0 \le a,b,c,d,e < p.$  Taking field norms, we find that 
$(N_{K\mid k_2}(u))^p = \pm\varepsilon^{3a}.$  Hence $p \mid a$ and $a=0.$  Likewise the fact that 
$(N_{K\mid k_3}(u))^p = \pm\mu_0^{2b}\mu_1^{2c}$ implies that $p$ divides both $b$ and $c$ and hence 
that $b=c=0.$  Choose prime ideals $\mathcal{L}_i \subset \o_K,$ $(i=1,2),$ 
with $\mathcal{L}_i \cap \Z = \ell_i \Z$ and such that 
$\rho_0 \equiv -3 \pmod{\mathcal{L}_i}$.  Then  
$\rho_0/\rho_3 \equiv 15 \pmod{\mathcal{L}_i},$ and $\rho_1/\rho_4 \equiv -6 \pmod{\mathcal{L}_i},$ and 
we deduce that 
$u^p \equiv (-1)^e 2^e3^{d+e}5^d \pmod{\mathcal{L}_i}.$  
Note that every Galois conjugate ideal $\mathcal{L}_i$ contains a conjugate of $3+\rho_0,$ but that no two such ideals can contain the same conjugate since this would imply that $\ell_i$ was a common divisor of $30n-143$ and $\disc(f_n)$ and so would contradict Lemma \ref{deg6_lem1}.  Hence $\ell_i$ completely splits in $K$ and the residual degree 
of $\mathcal{L}_i$ is equal to $1.$  Applying our power residue hypotheses we conclude that $d=e=0.$  But then $u^p = 1$ and so $u=1 \in U$; contradiction.  Hence $p \nmid \lbrack \o_K^\times : U \rbrack.$
\end{proof}

\begin{proposition}\label{deg6_prop3}
Let $n \equiv 2 \pmod{4}$ be an integer.  If there is an integer $y$ not divisible by $7$ satisfying $y^r = 30n-143=f_n(-3),$ 
and if for each prime $p\mid r$ the following conditions hold:
\begin{enumerate}
\item there are primes $\ell_i \mid 30n-143$ $(i=1,2),$ such that
\begin{enumerate}
\item $(2\,\vert\,\ell_1)_p = (3\,\vert\,\ell_1)_p = 1$ and $(5\,\vert\,\ell_1)_p\neq 1$
\item $(3\,\vert\,\ell_2)_p = (5\,\vert\,\ell_2)_p = 1$ and $(2\,\vert\,\ell_2)_p\neq 1$
\end{enumerate}
\item $w\rho_4/\rho_1$ fails to have a $p$th root in $K_n,$
\end{enumerate}
then $w\o_K = J^r$ where $\lbrack J \rbrack$ has order $r.$
\end{proposition}

\begin{proof}
From Proposition \ref{deg6_prop2} we deduce that $w\o_K = J^r$ for some ideal $J \in \mathcal{I}_K$ and 
so the order of $\lbrack J \rbrack$ is $\le r.$  If it was strictly $< r,$ then $w\o_K$ would be a prime power of 
a principal ideal for some prime $p \mid r.$  Supposing this to be the case and writing 
$w\o_K = (z\o_K)^p,$ for some $z \in K^\times,$ we have that $w = uz^p$ for some $u \in \o_K^\times.$ 
Since $p \nmid \lbrack \o_K^\times: U \rbrack,$ we may assume $u \in U.$  Since $p$ is odd, we may ignore 
any factor of $-1$ and assume that 
$$u = \varepsilon^a\mu_0^b\mu_1^c(\rho_0/\rho_3)^d(\rho_1/\rho_4)^e.$$
By absorbing any $p$th powers into $z^p,$ we may further assume that 
$a,b,c,d,$ and $e$ are integers in the interval $\lbrack 0, p).$
Since $N_{K\mid k_2}(w)=1,$ we deduce that 
$N_{K\mid k_2}(1/z)^p = N_{K \mid k_2}(u) = \pm\varepsilon^{3a}.$  Recalling 
that $p \mid r$ and that $(r,6)=1,$ we find that $p \mid a$ and so 
conclude that $a=0.$  Since $N_{K \mid k_3}(w)=1,$ we deduce that 
$N_{K\mid k_3}(1/z)^p = N_{K \mid k_3}(u) = \pm \mu_0^{2b}\mu_1^{2c}.$  Hence $p \mid b$ and $p \mid c$ and 
so $b=c=0.$  Thus $w = (\rho_0/\rho_3)^d(\rho_1/\rho_4)^ez^p.$  
Choose prime ideals $\mathcal{L}_i \subset \o_K,$ $(i=1,2),$ 
with $\mathcal{L}_i \cap \Z = \ell_i \Z$ and such that 
$\rho_0 \equiv -3 \pmod{\mathcal{L}_i}$.  Then $w \equiv 6 \pmod{\mathcal{L}_i},$ 
$\rho_3/\rho_0 \equiv 1/15 \pmod{\mathcal{L}_i},$ and $\rho_4/\rho_1 \equiv -1/6 \pmod{\mathcal{L}_i},$ and 
we deduce that 
$z^p \equiv (-1)^e 2^{1-e}3^{1-d-e}5^{-d} \pmod{\ell_i}.$  Applying 
our power residue hypotheses we conclude that $d=0,$ $e=1$ and thereby deduce that 
$w\rho_4/\rho_1$ has a $p$th root in $K.$  Since this contradicts our hypotheses 
$w\o_K$ is not a prime power of a principal ideal for any prime $p \mid r$ and 
$\lbrack J \rbrack$ has order $r.$
\end{proof}

\noindent To apply Proposition \ref{deg6_prop3}, we will need condition(s) sufficient to guarantee that  $w\rho_4/\rho_1$ 
fails to have a $p$th root in $K_n$ for each prime $p \mid r.$
The next proposition provides us with such a condition.

\begin{proposition}\label{deg6_prop4}
Let 
$$p_{n,r}(X) := \sum_{k=0}^6 a_kX^{rk},$$
where 
\begin{align*}
&a_0=a_6=(30n-143)^2\\
&a_1=a_5=-6(30n-143)^2\\
&a_2=a_4=-104149n^2 - 128700n - 12399357\\
&\;\;\;\;=-\frac{104149}{900}(30n-143)^2-\frac{16823807}{450}(30n-143)-\frac{13841287201}{900}\\
&a_3=-253298n^2 + 171600n - 25821164\\
&\;\;\;\;=-\frac{253298}{900}(30n-143)^2-\frac{16823807}{225}(30n-143)-\frac{13841287201}{450}\\
\end{align*}
If $p_{n,r}(X)$ is irreducible in $\Q\lbrack X \rbrack,$ 
then for each $p \mid r,$ $w\rho_4/\rho_1$ fails to have a $p$th root in $K_n.$
\end{proposition}

\begin{proof}
If there is a $z \in K_n^\times$ such that 
$z^p = w\rho_4/\rho_1$ with $p$ a prime dividing $r,$ then $X^r - w\rho_4/\rho_1$ factors non-trivially in 
$K_n\lbrack X \rbrack.$  Let $m(X) \in \Q\lbrack X \rbrack$ 
denote the minimal polynomial of $w\rho_4/\rho_1$ over $\Q.$  We note that 
$\deg m(X)=6,$ since $w\rho_4/\rho_1$ is a primitive element for $K_n/\Q.$  
(Note that $w\rho_4/\rho_1$ has trivial norm down to each proper subfield of $K_n.$  If $w\rho_4/\rho_1$ was contained 
in such a proper subfield we would be forced to conclude that $w = \rho_1/\rho_4.$  Using Equations (\ref{deg6_eqn4}),
we rewrite this as $14\rho^6+42\rho^5+28\rho^4-14\rho^3-42\rho^2-28\rho=0.$  Since $\rho \neq 0,$ we deduce that 
$\rho$ is a zero of a degree $5$ polynomial with integer coefficients.  Since $\rho$ is a primitive element for 
$K_n/\Q$ we've reached a contradiction and so conclude that $w\rho_4/\rho_1$ is not contained in any proper subfield 
of $K_n$ and so is a primitive element for the extension.)  Let $\zeta$ be a zero of $X^r - w\rho_1/\rho_4.$
Since $X^r - w\rho_1/\rho_4$ is 
reducible in $K_n\lbrack X \rbrack$ and since $\zeta^r = w\rho_1/\rho_4,$ we deduce that 
$\lbrack \Q(\zeta):\Q\rbrack = \lbrack \Q(\zeta):K_n\rbrack\lbrack K_n:\Q\rbrack < 6r.$  
But since $m(\zeta^r)=0,$ we deduce that $m(X^r)$ factors non-trivially in $\Q\lbrack X \rbrack$ 
(because otherwise $\lbrack \Q(\zeta):\Q \rbrack = \deg m(X^r) = 6r.$)  
Letting $d$ denote the least common multiple of the denominators of the coefficients of $m(X^r),$ 
it is clear that $d(m(X^r))$ factors non-trivially in $\Q\lbrack X \rbrack,$ since $m(X^r)$ factors non-trivially 
in $\Q\lbrack X \rbrack$ if and only if $d(m(X^r))$ does as well.  A straightforward calculation reveals that 
$d(m(X^r))$ is the polynomial, $p_{n,r}(X)$ given in the statement of the proposition.  Thus we have 
shown that the existence of a prime $p \mid r$ such that $w\rho_4/\rho_1 \in (K_n^\times)^p$ implies that 
$p_{n,r}(X)$ admits a non-trivial factorization in $\Q\lbrack X \rbrack.$
\end{proof}

\noindent Hilbert's Irreducibility Theorem \cite{hilbert} states that if $g(X,Y)$ is an irreducible polynomial in 
$\Z\lbrack X,Y \rbrack,$ then there are infinitely many $m \in \Z$ such that 
$g(X,m)$ is irreducible in $\Z\lbrack X \rbrack.$  In particular, if 
$g(X,Y)$ is an irreducible polynomial in $\Z\lbrack X,Y \rbrack$ and if for each $m \in \Z$ there 
is an $n_m \in \Z$ such that 
$$g(X,m)=p_{n_m,r}(X),$$ then Hilbert's Irreducibility Theorem 
implies that there are infinitely many $m \in \Z$ such that $p_{n_m,r}(X)$ is irreducible in $\Q\lbrack X \rbrack.$
We turn now to the construction of such a polynomial $g(X,Y) \in \Z\lbrack X,Y \rbrack.$ 

\begin{proposition}\label{deg6_lem3}
Let $r>0$ be an integer with $(r,6)=1$ and let $y_0,c_r \in \Z,$ where the latter is chosen such that 
$(c_r(-143))^r \equiv -143 \pmod{30}.$  Let 
$$
g(X,Y):= \sum_{k=0}^6 g_k(Y)X^{rk},
$$
where 
\begin{align*}
&g_0(Y)=g_6(Y)=(c_r(30(y_0+8q^2Y)-143))^{2r}\\
&g_1(Y)=g_5(Y)=-6g_0(Y)\\
&g_2(Y)=g_4(Y)=-\frac{104149}{900}g_0(Y)-\frac{16823807}{450}\sqrt{g_0(Y)}-\frac{13841287201}{900}\\
&g_3(Y)=-\frac{253298}{900}g_0(Y)-\frac{16823807}{225}\sqrt{g_0(Y)}-\frac{13841287201}{450}.\\
\end{align*}
Then $g(X,Y)$ is irreducible in $\Q\lbrack X,Y \rbrack$ and for each $m \in \Z,$ $g(X,m) = p_{n_m,r}(X),$ 
where $n_m$ is the unique integer satisfying 
$$(c_r(30(y_0+8q^2m)-143))^r=30n_m-143.$$
\end{proposition}

\noindent To establish the proposition, we first need the following 
\begin{lemma}\label{deg6_lem2}
Let $t \in \Q$ be such that $c_r(30(y_0+8q^2t)-143)=1.$  Then $g(X,t)$ is irreducible in $\Q\lbrack X \rbrack.$
\end{lemma}

\begin{proof}
Note that $g(X,t)=h(X^r),$ where 
$$
h(X) = X^6-6X^5-\frac{385417749}{25}X^4-\frac{770836748}{25}X^3-\frac{385417749}{25}X^2-6X+1.
$$
Since $25 \times h(X+1)$ is Eisenstein at $13,$ $h(X)$ is irreducible in $\Q\lbrack X \rbrack.$
Let $v$ be a zero of $h(X).$  If $g(X,t) = h(X^r)$ has a non-trivial factorization in $\Q\lbrack X \rbrack,$ 
then $X^r-v$ has a non-trivial factorization in $\Q(v)\lbrack X \rbrack.$  So it suffices to show that 
$X^r-v$ is irreducible in $\Q(v)\lbrack X \rbrack.$  By \cite[Chapter VI, \S 9, Theorem 9.1]{lang}, it suffices to show that $v \not\in (\Q(v))^p$ for each prime $p \mid r.$  Towards this end, 
we note that $5^6\;h\left(\frac{X}{5}\right)$ is a monic polynomial in $\Z\lbrack X \rbrack$ that has $5v$ as a zero.  
Hence $5v$ is an algebraic integer in $\Q(v).$  Since $v$ is not an algebraic integer in $\Q(v)$ it has some 
prime ideals in its denominator.  But since $5v$ is an algebraic integer in $\Q(v)$ and since $5$ splits completely 
in $\Q(v)$ (as is easily checked using PARI \cite{pari}, for example), $v$ must have the first power of some ideal in its 
denominator.  Hence $v \not\in (\Q(v))^p$ for every prime $p \mid r.$
\end{proof}

\begin{proof}[Proof of Proposition \ref{deg6_lem3}]
Suppose that $g(X,Y) = k(X,Y)j(X,Y)$ is a non-trivial factorization of $g(X,Y)$ in $\Q\lbrack X, Y\rbrack.$  If $k$ and $j$ each have degree $\ge 1$ in $X,$ then 
choosing $t \in \Q$ such that $c_r(30(y_0+8q^2t)-143)=1$ implies that 
$g(X,t)=k(X,t)j(X,t)$ is a non-trivial factorization of $g(X,t)$ in $\Q\lbrack X \rbrack$; 
a contradiction to Lemma \ref{deg6_lem2}.  We conclude that if $g(X,Y)$ factors non-trivially in 
$\Q\lbrack X,Y \rbrack,$ then one of the factors must have degree $0$ in $X;$ i.e., 
$g(X,Y)=k(Y)j(X,Y).$  Without loss of generality we may assume that 
$k(Y)$ is irreducible in $\Q\lbrack X,Y\rbrack.$  Hence $k(Y)$ divides each $g_i(Y)$ as well as 
$\sqrt{g_0(Y)}.$  We conclude that $k(Y)$ divides 
$900g_2(Y)+104149g_0(Y)+2(16823807)\sqrt{g_0(Y)} = -7^{12}.$  But $k(Y)$ is a non-trivial factor of 
$g(X,Y)$ in $\Q\lbrack X,Y \rbrack$ and so must have degree greater than $0$ in $Y.$  
This contradiction reveals that $g(X,Y)$ is irreducible in $\Q\lbrack X,Y \rbrack.$\\
\\
Note that for each $m \in \Z,$ 
$(c_r(30(y_0+8q^2m)-143))^r \equiv (-143c_r)^r \equiv -143 \pmod{30}.$  
Hence for each $m \in \Z$ there is an $n_m \in \Z$ such that 
$(c_r(30(y_0+8q^2m)-143))^r=30n_m-143$ and hence such that 
$g(X,m) = p_{n_m,r}(X).$  The uniqueness of $n_m$ is clear.
\end{proof}

\noindent We now prove

\begin{proposition}\label{deg6_prop6}
There is a $y \in \Z$ for which Proposition \ref{deg6_prop3} holds and hence a field $K_n \in \frak{F}$ for which 
\begin{enumerate}[label=(\roman*)]
\item there is an ideal $J \in \mathcal{I}_{K_n}$ whose ideal class, $\lbrack J \rbrack,$ has order $r.$
\item $N_{K \vert k_i}(J),$ $(i=2,3)$ is a principal ideal.
\end{enumerate}
\end{proposition}

\noindent Our proof makes use of the following Lemma

\begin{lemma}\label{deg6_lem4}
There is a constant $c_r$ with $c_r^r \equiv 7^{1-r} \pmod{30}$ and $(210q,c_r)=1$ 
and which is such that for each prime $p \mid r$ there are primes $\ell_i \mid c_r,$ 
$(i=1,2),$ for which
\begin{enumerate}
\item $(2\,\vert\,\ell_1)_p = (3\,\vert\,\ell_1)_p = 1$ and $(5\,\vert\,\ell_1)_p\neq 1$
\item $(3\,\vert\,\ell_2)_p = (5\,\vert\,\ell_2)_p = 1$ and $(2\,\vert\,\ell_2)_p\neq 1.$
\end{enumerate}
\end{lemma}

\begin{proof}
Let $L_1 = \Q(\zeta_p,2^{1/p},3^{1/p})$ and $L_2 = \Q(\zeta_{2p},5^{1/p}),$ where 
$\zeta_n$ denotes a primitive $n$th root of unity.  A prime $\ell$ totally splits in $L_1$ if and only if 
$\ell \equiv 1 \pmod{p}$ and $(2\,\vert\,\ell)_p = (3\,\vert\,\ell)_p = 1.$  Likewise a prime $\ell$ totally 
splits in $L_2$ if and only if $\ell \equiv 1 \pmod{2p}$ and $(5\,\vert\,\ell)_p=1.$  Since $L_2 \not\subset L_1,$ 
we may apply Bauer's Theorem to conclude that there are infinitely many primes $\ell$ that 
totally split in $L_1$ but not in $L_2.$  Hence there are infinitely many primes 
$\ell$ such that $(2\,\vert\,\ell)_p = (3\,\vert\,\ell)_p = 1$ and $(5\,\vert\,\ell)_p\neq 1.$  Choose such a prime 
$\ell_1$ and use the infinitude of such primes to ensure that $(\ell_1,210q)=1.$  Similarly, choose 
$\ell_2$ with $(\ell_2,210q)=1,$ and such that $(3\,\vert\,\ell_2)_p = (5\,\vert\,\ell_2)_p = 1$ and $(2\,\vert\,\ell_2)_p\neq 1.$
Repeat this procedure for each prime $p \mid r.$  Choose $a,b \in \Z$ such that 
$ra+4b=1$ and choose $s \in \N$ with $(s,7q)=1$ and such that 
$s\left(\prod_{p \mid r} \ell_1(p)\ell_2(p)\right) \equiv 7^{(1-r)a} \pmod{30}.$  Setting 
$c_r = s\left(\prod_{p \mid r} \ell_1(p)\ell_2(p)\right),$ we conclude that 
$c_r^r = s^r\left(\prod_{p \mid r} \ell_1(p)\ell_2(p)\right)^r \equiv 7^{(1-r)ar} \equiv \left(7^{1-r}\right)^{ra+4b} = 7^{1-r} \pmod{30}.$
\end{proof}

\begin{proof}[Proof of Proposition \ref{deg6_prop6}]
Choose $c_r$ according to Lemma \ref{deg6_lem4}, let $y_0 \in \Z,$ and 
let $y:= dc_r(30(y_0+8q^2m)-143),$ where $d$ satisfies the system of congruences 
$d \equiv 5/(c_r(30y_0-143)) \pmod{8}$ and $d \equiv 1 \pmod{30}.$
Then $y^r \equiv (-143c_r)^r \equiv (7c_r)^r \equiv 7 \pmod{30}.$  
Hence $y^r = 30n-143$ for some $n \in \Z.$  Furthermore, since $r$ is odd and 
since $y \equiv 5 \pmod{8},$ the congruence $5 \equiv y^r = 30n-143 \pmod{8}$ 
implies that %$30(n-2)-88=30n-148 \equiv 0 \pmod{8}$ and hence that 
$n \equiv 2 \pmod{4}.$
According to Lemma \ref{deg6_lem4}, 
for each prime $p \mid r$ there are primes dividing $c_r,$ and 
hence dividing $30n-143,$ that satisfy condition (1) of Proposition \ref{deg6_prop3}. 
Since $(dc_r(-143))^r \equiv -143 \pmod{30},$ Proposition 
\ref{deg6_lem3} and the discussion of Hilbert's Irreducibility Theorem preceding it reveal that $p_{n_m,r}(X)$ 
is irreducible in $\Q\lbrack X \rbrack$ for infinitely many $m \in \Z.$  Choosing such a pair $(m,n_m),$ 
Proposition \ref{deg6_prop4} reveals that $w\rho_4/\rho_1$ fails to have a $p$th root in $K_{n_m}$ for each 
prime $p \mid r.$  We see that 
for such a pair $(m,n_m)$ all of the hypotheses of Proposition \ref{deg6_prop3} are 
satisfied and we conclude that there is a field $K:=K_{n_m} \in \frak{F}$ for which 
$w\o_K = J^r$ where $\lbrack J \rbrack$ has order $r.$  Finally, since $J^r=w\o_K$ and since $N_{K \mid k_i}(w)=1$ $(i=2,3),$ it is clear that $N_{K \mid k_i}(J) = \o_{k_i}$ $(i=2,3).$
\end{proof}

\noindent At this point we have found a field $K \in \frak{F}$ that satisfies conditions 
$(i)$ and $(ii)$ of Proposition \ref{deg6_prop1}.  All that remains to establish this proposition is 
to show that $q$ ramifies in $k_2.$  This is accomplished through a judicious choice of $y_0$ in 
the following proof.

\begin{proof}[Proof of Proposition \ref{deg6_prop1}]
Since $\mathcal{P}(X)$ splits into linear factors$\pmod{q},$ we deduce that $(-108\,\vert\,q)_2=1.$  
Choose $n_0 \in \Z$ such that $n_0^2+108\equiv 0 \pmod{q}$ and $n_0^2+108 \not\equiv 0 \pmod{q^2}.$  
Use the fact that $\mathcal{P}(X)$ splits into linear factors $\pmod{q}$ to choose $b_0 \in \Z$ such that 
$b_0^r \equiv 30n_0-143 \pmod{q}.$  Note that $q \neq 7$ implies that $q \nmid b_0,$ since otherwise 
$q$ would be a prime common divisor of $30n_0-143$ and $n_0^2+108 \mid \disc(f_{n_0}(X)).$  Since 
$q \nmid b_0$ and since $q \nmid r,$ we conclude that we may lift $b_0$ to an $r$th root of 
$30n_0-143 \pmod{q^2}.$  Choose $c_r$ according to Lemma \ref{deg6_lem4} and note that 
since $q \nmid 30c_r,$ we may choose $y_0 \in \Z$ such that 
$b_0 \equiv c_r(30y_0-143) \pmod{q^2}.$\\
\\
Let $m \in \Z$ and let $y := dc_r(30(y_0+8q^2m)-143),$ where $d$ satisfies the system of congruences 
$d \equiv 5/(c_r(30y_0-143)) \pmod{8}$ and $d \equiv 1 \pmod{30q^2}.$
Then $y^r \equiv b_0^r \equiv 30n_0-143 \pmod{q^2}.$  Since 
$y^r \equiv -143 \pmod{30},$ we have $y^r=30n_m-143,$ for some $n_m \in \Z.$  
Since $n_m \equiv n_0 \pmod{q^2},$ we deduce 
that $n_m^2+108 \equiv 0 \pmod{q}$ and $n_m^2+108 \not\equiv 0 \pmod{q^2}.$  This implies that 
$q$ ramifies in $k_2 = \Q(\sqrt{n_m^2+108}) \subset K_{n_m},$ since $q$ divides the square-free part of 
$\disc(X^2-(n_m^2+108))$ and so divides $\disc(\o_{k_2}/\Z).$\\
\\
We conclude that for each $m \in \Z,$ the prime $q$ ramifies in $K_{n_m}.$  
To prove that one of these fields satisfies conditions 
$(i)$ and $(ii)$ of Proposition \ref{deg6_prop1}, one may 
copy the proof of Proposition \ref{deg6_prop6} verbatim beginning with the fourth sentence.
\end{proof}

% Quartics
\section{A family of quartic number fields.}

\noindent  In this section we apply the method introduced in the previous section to a family 
of quartic number fields.  Since the results and proofs of this section are highly similar to those presented in the previous section, much of this section will be abridged.\\
\\
\noindent Let $n \in \Z\smallsetminus \{0,\pm 3\}$ and let $K_n=\Q(\rho),$ where $\rho$
is a root of 
\begin{equation}
f_n(X) := X^4-nX^3-6X^2+nX+1.
\end{equation}
$K_n/\Q$ is a real cyclic quartic extension and was studied by M.-N. Gras in \cite{gras2}.
Let $G:= \gal(K_n/\Q).$ Then $G=\langle \sigma \mid \sigma^4=1 \rangle,$ 
where $\sigma : \rho \mapsto (\rho-1)/(\rho+1).$  The images of 
$\rho$ under $G,$ denoted
$\rho_i := \sigma^i(\rho),$ $(i=0,\ldots,3),$ are given by
\begin{equation}
\begin{gathered}
\rho_0 = \rho\\
\rho_2 = -1/\rho
\end{gathered} \qquad
\begin{gathered}
\rho_1 = (\rho-1)/(\rho+1)\\
\rho_3 = -(\rho+1)/(\rho-1).
\end{gathered}
\end{equation}  

\noindent Let $\frak{F}:=\{K_n \mid n \neq 0,\pm 3\}.$  
Let $\mathcal{I}_K$ denote the group of fractional ideals of $K\in \frak{F}$  
and for a fractional ideal $J \in \mathcal{I}_K,$ let $\lbrack J \rbrack$ denote its ideal class.  
Let $k_2=\Q(\sqrt{n^2+16})$ denote the unique quadratic subfield of $K.$
Then $\lbrack J \rbrack$ is said to be a \emph{relative ideal class} if it has trivial norm down to $k_2.$\\
\\
\noindent Note that the $\rho_i$ $(i=0,\ldots,3)$ are algebraic integers and, in fact, relative units of $K_n.$\\
\\
Our primary result is
\begin{theorem}\label{thm2}
Let $r>0$ be an odd integer.  There are infinitely many fields $K \in \frak{F}$ containing a relative ideal class of order $r.$
\end{theorem}
\noindent It is straightforward to see that $K_n = K_{-n}.$  However, the parametrization 
of $\frak{F}$ given by $n \mapsto K_n$ fails to be injective even when we restrict to $n>0.$ For example Lazarus \cite{lazarus} notes that $K_2 = K_{22}.$  Hence in proving Theorem 
\ref{thm2} it is insufficient to simply show that there are infinitely many $n$ such that $K_n$ contains a relative 
ideal class of order $r.$  Instead we put most of our effort into first proving 
\begin{proposition}
\noindent Let $r>0$ be an odd integer.  Let $q \in \Z$ be a prime such that 
$$
\mathcal{P}(X):=(X^r+7-6\sqrt{-16})(X^r+7+6\sqrt{-16}) \in \Z\lbrack X \rbrack
$$
splits into $\textmd{linear factors}\pmod{q}.$  
Assume that 
$q \nmid 5r.$  There is a field $K \in \frak{F}$ that satisfies the following three conditions:\label{thm1}
\begin{enumerate}[label=(\roman*)]
\item There is an ideal $J \in \mathcal{I}_K$ whose ideal class, $\lbrack J \rbrack,$ has order $r.$\label{1a}
\item $N_{K \mid k_2}(J)$ is a principal ideal.\label{1b}
\item $q$ ramifies in $k_2.$\label{1c}
\end{enumerate}
\end{proposition}

\noindent Once Proposition \ref{thm1} is established we have the following 

\begin{proof}[Proof of Theorem \ref{thm2}]
Replace $q \nmid 210r$ with $q \nmid 5r$ in the proof of Theorem \ref{deg6_thm1}.
\end{proof}

\noindent Let 
$$
w:= \frac{2+\rho_1}{2+\rho_3}
$$

\begin{proposition}\label{deg4_prop2}
If there is a $y \in \Z$ with $5\nmid y$ and such that $y^r = 6n-7=f_n(-2),$ then $w\o_{K_n} = J^r$ for some ideal $J \in \mathcal{I}_{K_n}.$
\end{proposition}

\noindent For the proof of this proposition we require the following easy lemma:

\begin{lemma} \label{deg4_lem1}
If a prime $p$ is a common divisor of $6n-7$ and $\disc(f_n),$ then $p = 5.$
\end{lemma}

\begin{proof}
The congruences $6n-7 \equiv 0 \pmod{p}$ and $\disc(f_n)= 4(n^2+16)^3\equiv 0 \pmod{p}$ hold simultaneously only if $p=5.$
\end{proof}

\begin{proof}[Proof of Proposition \ref{deg4_prop2}]
In the proof of Proposition \ref{deg6_prop2} above, replace $7 \nmid y$ with $5 \nmid y.$ 
Also restrict the indices $0 \le i \le 3$ and replace $(3+\rho_i)$ with $(2+\rho_i)$ and $f_n(-3)$ with $f_n(-2).$  Finally note that no restriction on the congruence class of $n$ 
is required, because $\rho$ and its conjugates are algebraic integers.
\end{proof}

\noindent Let $\varepsilon$ denote the fundamental unit of $k_2$ and let $U:=\langle -1, \varepsilon, \rho_0, \rho_1 \rangle \le \o_K^\times.$  We will need to know that $\lbrack \o_K^\times : U \rbrack < \infty.$  
It suffices to prove the following.

\begin{lemma}
The subset $S:=\{\varepsilon, \rho_0, \rho_1\} \subset \o_K^\times$ is multiplicatively independent.
\end{lemma}

\begin{proof}
Let $R$ denote the regulator of $S.$  Since $f_n(1)=-4$ and since 
$f_n(x) \to \infty$ as $x \to \infty,$ we know that $f_x(X)$ has a zero $>1.$  From 
the from of this zero Galois conjugates, we deduce that $f_n(X)$ has another positive 
zero in the interval $(0,1)$ and that the remaining two zeros of $f_n(X)$ are negative reciprocals 
of the two positive zeros.  
Taking $\rho_0$ to be the zero of $f_n(X)$ greater than $1,$ we find that
\begin{align*}
R &= \left\vert\det
\begin{pmatrix}
\log \rho_0 &\log \rho_1 &\log \varepsilon\\
\log \rho_1 &-\log \rho_0 &\log \vert \varepsilon^\sigma \vert\\
-\log \rho_0 &-\log \rho_1 &\log \varepsilon
\end{pmatrix} \right\vert\\
&= 2\log\varepsilon\; (\log^2\rho_0 + \log^2\rho_1).
\end{align*}
Since $\varepsilon > 1$ and $\rho_0 > 1,$ we conclude that $R > 0$ and 
that $S$ is a multiplicatively independent subset of $\o_K^\times.$
\end{proof}

\noindent We gently remind the reader that the notation (introduced in Section 1 above)
$(a\,\vert\,\ell)_p=1$ (resp. $(a\,\vert\,\ell)_p\neq 1$) means that 
$a$ is a $p$th power $\textmd{residue}\pmod{\ell}$ (resp. $a$ is a $p$th power $\textmd{nonresidue}\pmod{\ell}.$)

\begin{lemma}
Let $p$ be an odd prime.  
If a prime $\ell \mid 6n-7 = f_n(-2)$ with $\ell\neq 5$ is such that 
$(2\,\vert\,\ell)_p=1$ and $(3\,\vert\,\ell)_p\neq 1,$ then 
$p \nmid \lbrack \o_K^\times : U \rbrack.$
\end{lemma}

\begin{proof}
If $p \mid \lbrack \o_K^\times : U \rbrack,$ then there is a $u \in \o_K^\times \smallsetminus U$ such that 
$u^p = \pm \varepsilon^a\rho_0^b\rho_1^c.$  Without loss of generality we may take 
$u^p = \varepsilon^a\rho_0^b\rho_1^c$ with $0 \le a,b,c < p.$
Since $\ell$ divides $f_n(-2) = \prod_{i=0}^3(2+\rho_i),$ any prime ideal $\mathcal{L}$ dividing 
$\ell\o_K$ must contain one of the conjugates $2+\rho_0.$  Furthermore every 
conjugate $2+\rho_0$ must be contained in some prime ideal over $\ell,$ since the Galois group acts 
transitively on the set of conjugates of $2+\rho_0$ and since it takes prime ideals of $\o_K$ to 
prime ideals of $\o_K.$  Since $\ell \neq 5$ and since $\ell \mid 6n-7,$ we know that 
$\ell \nmid \disc(f_n),$ and so we deduce that every conjugate of $2+\rho_0$ is contained in 
a unique prime ideal over $\ell.$  Let $\Ell$ denote the unique prime ideal 
$\Ell \mid \ell\o_K$ such that $-2 \equiv \rho_0 \pmod{\Ell}.$  Modulo $\Ell$ this forces 
$\rho_1 \equiv 3,$ $\rho_2 \equiv 1/2,$ and $\rho_3 \equiv -1/3.$  Hence 
$$
(uu^\sigma)^p = (\pm 1)^a(\rho_0\rho_1)^b(\rho_1\rho_2)^c \equiv (\pm 1)^a(-1)^b2^{b-c}3^{b+c} 
\pmod{\Ell}
$$ and 
$$
(u^\sigma u^{\sigma^2})^p = (\pm 1)^a(\rho_1\rho_2)^b(\rho_2\rho_3)^c \equiv (\pm 1)^a(-1)^c2^{-(b+c)}3^{b-c} \pmod{\Ell}.$$  Applying the power residue hypotheses to these two congruences we find that 
$b \equiv c$ and $b \equiv -c \pmod{p}.$  Hence $b\equiv c \equiv 0 \pmod{p}$ and so $b=c=0.$  
Therefore $u^p = \varepsilon^a = (\varepsilon^a)^{\sigma^2} = (u^p)^{\sigma^2}.$  
Hence $u = u^{\sigma^2}.$  But then $u \in \o_K^\times \cap k_2 \subset U.$  We've reached a 
contradiction, and so are forced to conclude that $p \nmid \lbrack \o_K^\times : U \rbrack.$
\end{proof}

\begin{proposition}\label{deg4_prop3}
If there is an integer $y$ not divisible by $5$ satisfying $y^r = 6n-7=f_n(-2),$ 
and if for each prime $p\mid r$ the following conditions hold:
\begin{enumerate}
\item there are primes $\ell_i \mid $ $(i=1,2),$ such that
\begin{enumerate}
\item $(2\,\vert\,\ell_1)_p = 1$ and $(3\,\vert\,\ell_1)_p\neq 1$
\item $(3\,\vert\,\ell_2)_p = 1$ and $(2\,\vert\,\ell_2)_p\neq 1$
\end{enumerate}
\item $w/\rho_1$ fails to have a $p$th root in $K_n,$
\end{enumerate}
then $w\o_K = J^r$ where $\lbrack J \rbrack$ has order $r.$
\end{proposition}

\begin{proof}
The proof is nearly identical to that used in the proof of Proposition \ref{deg6_prop3} 
except that we assume the unit $u$ to have the form $u=\varepsilon^a\rho_0^b\rho_1^c.$  
Taking norms down to $k_2$ then allows us to deduce that $w = \rho_0^b\rho_1^cz^p.$  
Choosing the unique prime ideals $\mathcal{L}_i \mid \ell_i\o_K$ containing $\rho_0+2$ 
determines the classes of $\rho_0, \rho_1,$ and $w$ $\pmod{\mathcal{L}_i}$ 
and thereby allows us to deduce that $z^p \equiv (-1)^b2^b3^{1-c} \pmod{\ell_i}.$  
Applying the power residue hypotheses then forces $b=0$ and $c=1$ and so brings 
us to the contradiction $z^p = w/\rho_1.$
\end{proof}

\noindent To apply Proposition \ref{deg4_prop3}, we will need condition(s) sufficient to guarantee that  $w/\rho_1 \not\in (K_n^\times)^p$ for each prime $p \mid r.$
The next proposition provides us with such a condition.

\begin{proposition}\label{deg4_prop4}
Let 
$$p_{n,r}(X) := \sum_{k=0}^4 a_kX^{rk},$$
where 
\begin{align*}
&a_0=a_4=-a_2/6=6n-7\\
&a_1=-a_3=7n+96\\
\end{align*}
If $p_{n,r}(X)$ is irreducible in $\Q\lbrack X \rbrack,$ 
then $w/\rho_1$ fails to have a $p$th root in $K_n$ for every prime $p \mid r.$
\end{proposition}

\begin{proof}
If there is a $z \in K_n^\times$ such that 
$z^p = w/\rho_1$ with $p$ a prime dividing $r,$ then $X^r - w/\rho_1$ factors non-trivially in 
$K_n\lbrack X \rbrack.$  Let $m(X) \in \Q\lbrack X \rbrack$ 
denote the minimal polynomial of $w/\rho_1$ over $\Q.$  We note that 
$\deg m(X)=4,$ since $w/\rho_1$ is a primitive element for $K_n/\Q.$   
(Note that $N_{K_n \mid k_2}(w/\rho_1)=-1$, so the real number $w/\rho_1$ can't be contained in a 
proper subfield of $K_n.$)  Let $\zeta$ be a zero of $X^r - w/\rho_1.$
Since $X^r - w/\rho_1$ is 
reducible in $K_n\lbrack X \rbrack$ and since $\zeta^r = w/\rho_1,$ we deduce that 
$\lbrack \Q(\zeta):\Q\rbrack = \lbrack \Q(\zeta):K_n\rbrack\lbrack K_n:\Q\rbrack < 4r.$  
But since $m(\zeta^r)=0,$ we deduce that $m(X^r)$ factors non-trivially in $\Q\lbrack X \rbrack$ 
(because otherwise $\lbrack \Q(\zeta):\Q \rbrack = \deg m(X^r) = 4r.$)  
Since $p_{n,r}(X)=a_0m(X^r),$ we deduce that $p_{n,r}(X)$ 
factors non-trivially in $\Q\lbrack X \rbrack.$
Thus we have shown that the existence of a prime $p \mid r$ such that $w/\rho_1 \in (K_n^\times)^p$ implies that 
$p_{n,r}(X)$ admits a non-trivial factorization in $\Q\lbrack X \rbrack.$
\end{proof}

\noindent We now turn to the construction of a polynomial $g(X,Y)$ 
that is irreducible in $\Z\lbrack X, Y \rbrack$ and which has 
the property that for each $m \in \Z$ there is an $n_m \in \Z$ such that 
$g(X,m)=p_{n_m,r}(X).$  Applying Hilbert's Irreducibility Theorem \cite{hilbert}, as was done in Section 1, will then allow us to conclude that there are infinitely many $m \in \Z$ such that 
$p_{n_m,r}(X)$ is irreducible in $\Q\lbrack X \rbrack.$

\begin{proposition}\label{deg4_lem3}
Let $r>0$ be an odd integer and let $y_0,c_r \in \Z,$ where the latter is chosen such that 
$c_r \equiv 1 \pmod{6}.$  
Let 
$$
g(X,Y):= \sum_{k=0}^3 g_k(Y)X^{rk},
$$
where 
\begin{align*}
&g_0(Y)=g_4(Y)=-g_2(Y)/6 = (c_r(30(y_0+q^2Y)-1))^r\\
&g_1(Y)=-g_3(Y)=7(g_0(Y)+7)/6+96.
\end{align*}
Then $g(X,Y)$ is irreducible in $\Q\lbrack X,Y \rbrack$ and for each $m \in \Z,$ $g(X,m) = p_{n_m,r}(X),$ 
where $n_m$ is the unique integer satisfying $(c_r(30(y_0+q^2m)-1))^r=6n_m-7.$
\end{proposition}

\noindent To establish the proposition, we first need the following 
\begin{lemma}\label{deg4_lem2}
Let $t \in \Q$ be such that $c_r(30(y_0+q^2t)-1)=-1.$  Then $g(X,t)$ is irreducible in $\Q\lbrack X \rbrack.$
\end{lemma}

\begin{proof}
Note that $g(X,t)=-f_{-103}(X^r).$  
For this proof, let $f(X):=f_{-103}(X)$ and $K:= K_{-103}=K_{103}.$  If $g(x,t)=0,$ then 
$\prod_{i=0}^3(x^r-\rho_i) = f(x^r)=0,$ and so $x$ is zero of 
$X^r - \rho_i \in K\lbrack X \rbrack$ for some index $i.$  Let $\rho := \rho_i$ and 
note that it will suffice to show that $X^r - \rho$ is irreducible in $K\lbrack X \rbrack.$  
By \cite[Chapter VI, \S 9, Theorem 9.1]{lang}
it is sufficient to show that $\rho \not\in (K^\times)^p$ for each prime $p \mid r.$  
Using PARI \cite{pari} we find that a fundamental system of units of $K$ is $\{u_i\}_{i=1}^3,$ where
\begin{align*}
u_1&=(1/125)\rho^3+(101/125)\rho^2-(208/125)\rho+63/125\\
u_2&=(1/250)\rho^3+(53/125)\rho^2+(156/125)\rho+83/250\\
u_3&=(1/250)\rho^3+(48/125)\rho^2-(364/125)\rho-207/250.
\end{align*}
In terms of these units, $\rho=u_1^2u_2^3.$  If $z \in K^\times$ is such that $z^p = \rho$ for some prime $p,$ 
then $z\in \o_K^\times$; say $z = \prod_{i=1}^3u_i^{a_i}.$  So 
$\left(\prod_{i=1}^3u_i^{a_i}\right)^p = u_1^2u_2^3.$ Since there is no prime $p$ dividing both $2$ and $3,$ 
we conclude that $\rho \not\in (K^\times)^p$ for each prime $p \mid r.$
\end{proof}

\begin{proof}[Proof of Proposition \ref{deg4_lem3}]
Suppose that $g(X,Y) = k(X,Y)j(X,Y)$ is a non-trivial factorization of $g(X,Y)$ in $\Q\lbrack X, Y\rbrack$ 
and then by choosing $t \in \Q$ such that $c_r(30(y_0+q^2t)-1)=-1$ proceed 
as in the proof of Lemma \ref{deg6_lem3} to show that one of the factors must have 
degree $0$ in $X$; say $g(X,Y)=K(Y)j(X,Y).$
Without loss of generality assume that 
$k(Y)$ is irreducible in $\Q\lbrack X,Y\rbrack.$  Hence $k(Y)$ divides each $g_i(Y)$  
and we conclude that $k(Y)$ divides 
$6g_2(Y)-7g_1(Y) = 625.$  But $k(Y)$ is a non-trivial factor of 
$g(X,Y)$ in $\Q\lbrack X,Y \rbrack$ and so must have degree greater than $0$ in $Y.$  
This contradiction reveals that $g(X,Y)$ is irreducible in $\Q\lbrack X,Y \rbrack.$  For each $m \in \Z,$ 
$(c_r(30(y_0+q^2m)-1))^r \equiv (-c_r)^r \equiv -1 \pmod{6}.$  
Hence for each $m \in \Z$ there is an $n_m \in \Z$ such that 
$(c_r(30(y_0+q^2m)-1))^r=6n_m-7$ and hence such that 
$g(X,m) = p_{n_m,r}(X).$  The uniqueness of $n_m$ is clear.
\end{proof}

\noindent We now prove the following.

\begin{proposition}\label{deg4_prop6}
There is a $y \in \Z$ for which Proposition \ref{deg4_prop3} holds and hence a field $K_n \in \frak{F}$ for which 
\begin{enumerate}[label=(\roman*)]
\item there is an ideal $J \in \mathcal{I}_{K_n}$ whose ideal class, $\lbrack J \rbrack,$ has order $r.$
\item $N_{K \vert k_2}(J)$ is a principal ideal.
\end{enumerate}
\end{proposition}

\noindent Our proof makes use of the following Lemma.

\begin{lemma}\label{deg4_lem4}
There is a constant $c_r \equiv 1 \pmod{6}$ with $(c_r,5q)=1$ 
and which is such that for each prime $p \mid r$ there are primes $\ell_i \mid c_r,$ 
$(i=1,2),$ for which
\begin{enumerate}
\item $(2\,\vert\,\ell_1)_p = 1$ and $(3\,\vert\,\ell_1)_p\neq 1$
\item $(3\,\vert\,\ell_2)_p = 1$ and $(2\,\vert\,\ell_2)_p\neq 1.$
\end{enumerate}
\end{lemma}

\begin{proof}
Let $L_1 = \Q(\zeta_p,2^{1/p})$ and $L_2 = \Q(\zeta_{2p},3^{1/p}),$ 
and then for each prime $p$ dividing $r$ use Bauer's Theorem as in the proof of Lemma 
\ref{deg6_lem4} to choose primes $\ell_1(p)$ (resp. $\ell_2(p)$) satisfying condition (1) (resp. 
condition (2)) of the lemma with each $\ell_i \nmid 30q.$
Then set $c_r = s\left(\prod_{p \mid r} \ell_1(p)\ell_2(p)\right),$ where 
$s$ is chosen to ensure that $c_r \equiv 1 \pmod{6}.$
\end{proof}

\begin{proof}[Proof of Proposition \ref{deg4_prop6}]
Choose $c_r$ according to Lemma \ref{deg4_lem4}, let $y_0 \in \Z,$ and 
let $y:= c_r(30(y_0+q^2m)-1).$
Then $y^r \equiv (-c_r)^r \equiv -1 \pmod{6}.$  
Hence $y^r = 6n-7$ for some $n \in \Z.$  According to Lemma \ref{deg4_lem4}, 
for each prime $p \mid r$ there are primes dividing $c_r,$ and 
hence dividing $6n-7,$ that satisfy condition (1) of Proposition \ref{deg4_prop3}. 
Since $c_r \equiv 1 \pmod{6},$ Proposition 
\ref{deg4_lem3} and the discussion of Hilbert's Irreducibility Theorem preceding it reveal that $p_{n_m,r}(X)$ 
is irreducible in $\Q\lbrack X \rbrack$ for infinitely many $m \in \Z.$  Choosing such a pair $(m,n_m),$ 
Proposition \ref{deg4_prop4} reveals that $w/\rho_1$ fails to have a $p$th root in $K_{n_m}$ for each 
prime $p \mid r.$  We see that 
for such a pair $(m,n_m)$ all of the hypotheses of Proposition \ref{deg4_prop3} are 
satisfied and we conclude that there is a field $K:=K_{n_m} \in \frak{F}$ for which 
$w\o_K = J^r$ where $\lbrack J \rbrack$ has order $r.$  Finally, since $J^r=w\o_K$ and since $N_{K \mid k_2}(w)=1,$ it is clear that $N_{K \mid k_2}(J) = \o_{k_2}.$
\end{proof}

\noindent At this point we have found a field $K \in \frak{F}$ that satisfies conditions 
$(i)$ and $(ii)$ of Proposition \ref{thm1}.  All that remains to establish this proposition is 
to show that $q$ ramifies in $k_2.$  This is accomplished through a judicious choice of $y_0$ in 
the following proof.

\begin{proof}[Proof of Proposition \ref{thm1}]
Since $\mathcal{P}(X)$ splits into linear factors$\pmod{q},$ we deduce that $(-16\,\vert\,q)_2=1.$  
Choose $n_0 \in \Z$ such that $n_0^2+16\equiv 0 \pmod{q}$ and $n_0^2+16 \not\equiv 0 \pmod{q^2}.$  
Use the fact that $\mathcal{P}(X)$ splits into linear factors $\pmod{q}$ to choose $b_0 \in \Z$ such that 
$b_0^r \equiv 6n_0-7 \pmod{q}.$  Note that $q \neq 5$ implies that $q \nmid b_0,$ since otherwise 
$q$ would be a prime common divisor of $6n_0-7$ and $n_0^2+16 \mid \disc(f_{n_0}(X)).$  Since 
$q \nmid b_0$ and since $q \nmid r,$ we conclude that we may lift $b_0$ to an $r$th root of 
$6n_0-7 \pmod{q^2}.$  Choose $c_r$ according to Lemma \ref{deg4_lem4} and note that 
since $q \nmid 30c_r,$ we may choose $y_0 \in \Z$ such that 
$b_0 \equiv c_r(30y_0-1) \pmod{q^2}.$\\
\\
Let $m \in \Z$ and let $y := c_r(30(y_0+q^2m)-1).$  
Then $y^r \equiv b_0^r \equiv 6n_0-7 \pmod{q^2}.$  Since 
$y^r \equiv -1 \pmod{6},$ we have $y^r=6n_m-7,$ for some $n_m \in \Z.$  
Since $n_m \equiv n_0 \pmod{q^2},$ we deduce 
that $n_m^2+108 \equiv 0 \pmod{q}$ and $n_m^2+108 \not\equiv 0 \pmod{q^2}.$  This implies that 
$q$ ramifies in $k_2 = \Q(\sqrt{n_m^2+16}) \subset K_{n_m},$ since $q$ divides the square-free part of 
$\disc(X^2-(n_m^2+16))$ and so divides $\disc(\o_{k_2}/\Z).$\\
\\
We conclude that for each $m \in \Z,$ the prime $q$ ramifies in $K_{n_m}.$  
To prove that one of these fields satisfies conditions 
$(i)$ and $(ii)$ of Proposition \ref{thm1}, one may 
copy the proof of Proposition \ref{deg4_prop6} verbatim beginning with the fourth sentence.
\end{proof}

% Non-Galois Cubics
\section{A family of non-Galois cubic number fields.}

\noindent Let $K_n=\Q(\rho),$ where $\rho$
is a root of 
\begin{equation}
f_n(X) := X^3+nX^2+nX-1.
\end{equation}
Since $f_n(\rho)=0$ if and only if $f_{-n}(1/\rho)=0,$ we have that $K_n = K_{-n}$ and so we 
restrict our attention to $n\ge 0.$  When $n>0,$ we find 
that $f_n(\pm 1) \neq 0$ and so conclude that $f_n(X)$ is irreducible in $\Q\lbrack X \rbrack$ 
(since $\pm 1$ are the only possible zeros of $f_n(X)$ in $\Q$).  Since this clearly doesn't hold when 
$n=0,$ we restrict our attention to the cubic extensions $K_n/\Q$ corresponding to $n > 0.$  
When $n \ge 5,$ $f_n(X)$ changes sign in the intervals $(1-n,2-n),$ $(-1-4/n,-1-1/n),$ and 
$(1/(n+1),1/n).$  In particular, $f_n(X)$ has real 
zeros, two of which are $< -1,$ while the third is $>0.$  When $\vert n \vert < 5,$ 
$f_n(X)$ has some complex zeros and so we restrict our attention to the real cubic extensions 
$K_n/\Q$ corresponding to $n \ge 5.$  
Note that $\disc(f_n(X))=n^4-18n^2-27 = (n^2-9)^2-108.$  Since $108$ is a difference 
of two squares in only two ways and since neither $12$ or $28$ is of the form $n^2-9,$ we may conclude that $\disc(f_n(X))$ is not a square 
in $\Q,$ and hence that the Galois group of $f_n(X)$ is $S_3.$  
Thus $K_n/\Q$ is a non-Galois real cubic extension.
We place one last restriction on the parameter $n$; we restrict $n$ to have odd parity. 
This ensures the existence of the explicit unit $\mu \in \o_K^\times$, distinct from $\rho$, and 
introduced in Lemma \ref{ng_lem3} below.\\
\\
Let $\frak{F} := \{K_n \mid n \textmd{ is an odd integer } \ge 5\}.$
Let $\mathcal{I}_K$ denote the group of fractional ideals of $K\in \frak{F}$  
and for a fractional ideal $J \in \mathcal{I}_K,$ let $\lbrack J \rbrack$ denote its ideal class.  Let 
$h_K$ denote the class number of $K.$
Our primary result is
\begin{theorem}\label{ng_thm1}
Let $r \not \equiv 0 \pmod{3}$ be a positive integer.   There are infinitely many fields $K \in \frak{F}$ containing an ideal class of order $r.$
\end{theorem}
\noindent We do not know whether or not the parametrization 
of $\frak{F}$ given by $n \mapsto K_n$ is injective and so in proving Theorem 
\ref{ng_thm1} we do not attempt to show that there are infinitely many $n$ such that $K_n$ contains an 
ideal class of order $r.$  Instead we put most of our effort into first proving 
\begin{proposition}\label{ng_prop1}
Let $r \not \equiv 0 \pmod{3}$ be an integer $>1.$  Then there is a field $K \in \frak{F}$ and an ideal $J \in \mathcal{I}_K$ whose 
ideal class, $\lbrack J \rbrack,$ has order $r.$
\end{proposition}

\noindent Once Proposition \ref{ng_prop1} is established we have the following 

\begin{proof}[Proof of Theorem \ref{ng_thm1}]
Since every class group contains the trivial class, the Theorem will hold 
for $r=1$ if it holds for some $r>1,$ so assume $r>1.$  The proposition tells us that the set 
$S$ consisting of all $K \in \frak{F}$ containing an ideal class of order $r$ is nonempty.  
If $S$ was a finite set, then there would be a natural number $\alpha \ge 1$ and a field $K \in S$ with 
class number $h_K$ satisfying $r^\alpha \mid h_K$ and $r^{\alpha+1} \nmid h_{K^\prime}$ for each 
$K^\prime \in S.$  By the proposition there is a field $L \in \frak{F}$ and an ideal $J$ of $L$ whose class 
has order $r^{\alpha+1}.$  Hence $L$ also contains a class of order $r,$ but since $L \not\in S$ we've reached 
a contradiction.  We conclude that $S$ is infinite.
\end{proof}

\noindent To establish Proposition \ref{ng_prop1}, we begin by establishing sufficient conditions for the 
ideal $(2-\rho)\o_K \in \mathcal{I}_K$ to factor as an $r$th power 
of an ideal $J \in \mathcal{I}_K.$  This is accomplished in Proposition \ref{ng_prop2} and 
$\lbrack J \rbrack,$ when $J$ exists, serves as our candidate for an ideal class of order $r.$  
Note that in the previous two sections we accomplished this task with the help of the Galois group, 
$G,$ of our extensions.  For if $y^r=f_n(m)=\prod_{i=0}^d(m-\rho_i),$ where $d=\deg(f),$ then 
$G$ acts transitively on factors $(m-\rho_i)$ and \emph{takes prime ideals of $K$ to prime ideals of $K.$}  
Since our cubic extension $K/\Q$ is non-Galois, this technique is no longer available to us.  Indeed, 
if $\sigma \in \gal(L/\Q)\smallsetminus \gal(L/K),$ where $L$ is the Galois closure of $K/\Q,$ 
then $\sigma \vert_K$ carries a prime ideal of $K$ to a prime ideal of an isomorphic but \emph{distinct} 
cubic extension $K^\prime/\Q.$  At its core, however, this technique requires us to show that 
a prime ideal $\frak{P}$ dividing $(m-\rho_0)\o_K$ doesn't also divide 
$\prod_{i=1}^d(m-\rho_i)\o_K,$ because this ensures that if $\nu(\frak{P})$ is the exact power of $\frak{P}$ 
dividing $(m-\rho_0),$ then it is also the exact power of this ideal dividing $y^r=f_n(m)$ and hence that $r \mid \nu(\frak{P}).$
This sufficiently motivates the following pair of lemmas building up to Proposition \ref{ng_prop2} and in regard to these we add that 
\begin{equation}
6n+7 = f_n(2)=\prod_{i=0}^2(2-\rho_i).
\end{equation}

\begin{lemma}\label{ng_lem1}
Let $\rho_i$ $(i=0,1,2)$ denote the roots of $f_n(X)$ with $\rho := \rho_0.$  
Then $$
(2-\rho_1)(2-\rho_2)=\rho^2+(2+n)\rho+(4+3n) \in \o_{K_n}.
$$
\end{lemma}

\begin{proof}
Since $f_n(X) = X^3+nX^2+nX-1,$ we have that $\sum_{i=0}^2 \rho_i = -n$ and 
$\prod_{i=0}^2 \rho_i = 1.$
\end{proof}

\begin{lemma}\label{ng_lem2}
If a prime $p \mid (6n+7)$ and $\frak{P} \mid p\o_K$ is a prime common divisor of $(2-\rho)\o_K$ and 
$(2-\rho_1)(2-\rho_2)\o_K,$ then $p=37.$
\end{lemma}

\begin{proof}
Applying Lemma \ref{ng_lem1}, we find that $5n+12 = (2-\rho_1)(2-\rho_2) + (2-\rho)(\rho+4+n) \in \frak{P}.$  
We thus have that $5n \equiv -12 \pmod{p}$ and $6n \equiv -7 \pmod{p}.$  These hold simultaneously only if 
$p=37.$
\end{proof}

\begin{proposition}\label{ng_prop2}
If $37 \nmid (6n+7)$ and $6n+7=y^r$ for some $y \in \Z,$ then there is an ideal $J \in \mathcal{I}_K$ such that 
$(2-\rho)\o_K=J^r.$
\end{proposition}

\begin{proof}
Let $\frak{P}$ be a prime ideal of $\o_K$ with $\frak{P}\cap \Z := p\Z$ and with $\frak{P} \mid (2-\rho)\o_K.$  
Then $6n+7 = f_n(2) = \prod_{i=0}^2 (2-\rho_i) \in \frak{P}.$  Hence $p \mid 6n+7$ and since 
$37 \nmid 6n+7,$ we find that $\frak{P} \nmid (2-\rho_1)(2-\rho_2)\o_K.$  Let 
$\frak{P}^{\nu_\frak{P}}$ denote the exact power of $\frak{P}$ dividing $(2-\rho)\o_K.$  
Then it must also be the exact 
power of $\frak{P}$ dividing $y^r\o_K = (6n+7)\o_K = \prod_{i=0}^2 (2-\rho_i)\o_K.$  Hence $r \mid \nu_\frak{P}.$  
We now set 
$$
J:= \prod_{\substack{\textmd{prime ideals } \\ \frak{P} \mid (2-\rho)\o_K}}\frak{P}^{\nu_\frak{P}/r},
$$
and note that $J^r = (2-\rho)\o_K.$
\end{proof}

\noindent In Proposition \ref{ng_prop3} we add sufficient conditions to those of 
Proposition \ref{ng_prop2} to ensure that $\lbrack J \rbrack$ has order $r.$  We turn first to a sequence of 
lemmas that we use in our proof of Proposition \ref{ng_prop3}.  We gently remind the reader that 
$n$ is odd.

\begin{lemma} \label{ng_lem3}
The algebraic integer
$$
\mu := \left(\frac{n+3}{2}\right)(1+\rho)+\rho^2
$$
is a unit in $\o_K$ with $N_{K \mid \Q}(\mu) = 1.$
\end{lemma}

\begin{proof}
Let $a,b,c \in \Z.$  A straightforward calculation shows that 
$N_{K\mid \Q}(a+b\rho+c\rho^2)=(a^2c+ac^2-abc)n^2-(2a^2c+a^2b-2ac^2-ab^2-bc^2+b^2c)n+(a^3+b^3+c^3-3abc).$  Substituting $a=b=(n+3)/2$ and $c=1$ into this expression yields 
$N_{K \mid \Q}(\mu) = 1.$  Since $\mu\in \o_K,$ the lemma is established.
\end{proof}

\begin{lemma}\label{ng_lem4}
The unit $\mu \in \o_K^\times$ is totally positive, while the unit $\rho \in \o_K^\times$ is neither 
totally positive nor totally negative.
\end{lemma}

\begin{proof}
The unit $\rho$ is a zero of $f_n(X).$  As mentioned in 
the introduction, two of the zeros of $f_n(X)$ are $<-1,$ while the third is $>0.$  In particular, $\rho$ is neither totally positive nor totally negative.  
Since $\rho_i^3+n\rho_i^2+n\rho_i-1=0,$ $(i=0,1,2),$ we see that 
$n = (1-\rho_i^3)/(\rho_i^2+\rho_i).$  Substituting this expression for $n$ into 
$\mu_i = \left(\frac{n+3}{2}\right)(1+\rho_i) + \rho_i^2$ yields 
\begin{equation}\label{ng_eqn1}
\mu_i = (\rho_i+1)^3/(2\rho_i)
\end{equation}
The location of the $\rho_i$ now confirms that each $\mu_i$ is positive.
\end{proof}

\noindent Let $U:=\langle -1, \mu, \rho \rangle \le \o_K^\times.$  We will need to know that $\lbrack \o_K^\times : U \rbrack < \infty.$  It will suffice to prove the following.

\begin{lemma}
The subset $S:=\{\mu, \rho\} \subset \o_K^\times$ is multiplicatively independent.
\end{lemma}

\begin{proof}
When $n \ge 7,$ the zeros of $f_n(X)$ satisfy $\rho_0 \in (1-n,2-n),$ 
$\rho_1 \in (-1-3/n,-1-1/n)$ and $\rho_2 \in (1/(n+1),1/n).$  The regulator, $R,$ of $S$ is 
$$
R = \bigg\vert \log \vert \mu_0 \vert \log \vert \rho_1 \vert - \log \vert \rho_0 \vert \log \vert \mu_1 \vert\bigg\vert.
$$

\noindent For $n \ge 7,$  we have that $1+1/n < \vert \rho_1 \vert < 1+3/n$ and 
$
1/n < \vert \rho_1+1 \vert < 3/n.
$
Using the relation given in Equation (\ref{ng_eqn1}), we conclude that 
\begin{align*}
1 &> \frac{27}{2n^2(n+1)} = (3/n)^3/(2(1+1/n))\\
&> \vert \mu_1 \vert = \vert \rho_1+1 \vert^3/(2\vert \rho_1 \vert)\\
&> (1/n)^3/(2(1+3/n)) = \frac{1}{2n^2(n+3)}.
\end{align*}
Since $1-n < \rho_0 < 2-n,$ we have that $n-2 < \vert \rho_0 \vert < n-1$ and
$
n-3 < \vert \rho_0+1 \vert < n-2.  
$
Using Equation (\ref{ng_eqn1}) again, we conclude that 
$$
(n-2)^2/2 = (n-2)^3/(2(n-2)) > \vert \mu_0 \vert = \vert \rho_0+1 \vert^3/(2\vert \rho_0 \vert) > (n-3)^3/(2(n-1))  > 1.
$$

\noindent Hence
\begin{align*}
R &= \log \vert \mu_0 \vert \log \vert \rho_1 \vert - \log \vert \rho_0 \vert \log \vert \mu_1 \vert\\
&\ge
\log \left(\frac{(n-3)^3}{2(n-1)}\right) \log(1+1/n) - \log\left(n-2\right)\log\left(\frac{27}{2n^2(n+1)}\right)\\
&\ge \log(n-2)\log\left(\frac{2n^2(n+1)}{27}\right)\\
&> 0,
\end{align*}
and we conclude that $S$ is a multiplicatively independent subset of $K_n$ for $n \ge 7.$  
It may be verified by hand that $R \neq 0$ when $n=5,$ so that $\{\rho,\mu\}$ is 
a multiplicatively independent subset of $K_n$ for all $n$ under consideration.
\end{proof}

\begin{lemma}\label{ng_lem5}
Let $\frak{P}$ be a prime ideal of $\o_K$ with $\frak{P}\cap \Z = p\Z.$  If $\rho \equiv 2 \pmod{\frak{P}}$ 
and $p \nmid \disc(f_n),$ then $\lbrack \o_K/\frak{P} : \Z/p\Z \rbrack = 1.$
\end{lemma}

\begin{proof}
Noting that 
$f_n(X) = (X-\rho)(X^2+(\rho+n)X+\rho^2+n\rho+n)$ and making use of the fact that 
$\rho \equiv 2 \pmod{\frak{P}}$ along with the monomorphism 
$\frac{\Z}{p\Z}\lbrack X \rbrack \hookrightarrow \frac{\o_K}{\frak{P}}\lbrack X \rbrack,$ we see that 
$f_n(X)$ factors as 
$(X-2)(X^2+(n+2)X+4+3n)$ in $(\Z/p\Z)\lbrack X \rbrack.$  Applying Dedekind's Factorization Theorem, 
we conclude that $\lbrack \o_K/\frak{P}:\Z/p\Z\rbrack = 1.$
\end{proof}

\begin{lemma}\label{ng_lem6}
If a prime $p$ is a common divisor of $6n+7$ and $\disc(f_n),$ then $p=37$ or $p=47.$
\end{lemma}

\begin{proof}
The congruences $n^4-18n^2-27 = \disc(f_n) \equiv 0 \pmod{p}$ and 
$6n+7 \equiv 0 \pmod{p}$ hold simultaneously only if $p=37$ or $p=47.$
\end{proof}

\begin{lemma} \label{ng_lem7}
Let $y \in \Z$ not divisible by 
either $37$ or $47$ and such that $y^r = 6n+7,$ 
and let $p$ be a prime dividing $r.$  If there are primes $\ell_i \mid (6n+7),$ $(i=1,2),$ and such that 
\begin{enumerate}[label=(\roman*)]
\item $(2\,\vert\,\ell_1)_p=1$ and $(3\,\vert\,\ell_1)_{p}\neq 1,$ \label{ng_lem7b}
\item $(2\,\vert\,\ell_2)_p \neq 1$ and $(3\,\vert\,\ell_2)_{p}\neq 1,$ \label{ng_lem7c}
\end{enumerate}
then $p \nmid \lbrack \o_K^\times : U \rbrack.$
\end{lemma}

\begin{proof}
Choose prime ideals $\mathcal{L}_i \subset \o_K,$ $(i=1,2),$ with $\mathcal{L}_i \cap \Z = \ell_i \Z$ 
and such that $\rho \equiv 2 \pmod{\mathcal{L}_i}.$  Then by 
equation \ref{ng_eqn1} $\mu \equiv 27/4 \pmod{\mathcal{L}_i}.$  
If $p \mid \lbrack \o_K^\times : U \rbrack,$ then there is a unit $u \in \o_K^\times \smallsetminus U$ 
and $a,b\in \Z$ with $0 \le a,b < p$ and such that $u^p = \pm \rho^a\mu^b.$  We remark that 
Lemma \ref{ng_lem6} guarantees that $\ell_i \nmid \disc(f_n)$ and so 
Lemma \ref{ng_lem5} then guarantees that $\lbrack \o_K/\mathcal{L}_i : \Z/\ell_i\Z \rbrack = 1.$\\
\\
If $p \neq 2,$ then by making use of the fact that $\lbrack \o_K/\mathcal{L}_i : \Z/\ell_i\Z \rbrack = 1$ 
along with the congruences for $\rho$ and $\mu,$ we find that $(2^{a-2b}3^{3b}\,\vert\,\ell_i)_p = 1.$  
We remind the reader that $p \neq 3,$ since $p \mid r$ and since $r \not\equiv 0 \pmod{3}.$  
Hence upon applying condition \ref{ng_lem7b} we conclude that $p \mid b$ and hence that $b=0.$  Setting $b=0$ 
we conclude that $(2^a\,\vert\,\ell_i)_p = 1.$  Applying condition \ref{ng_lem7c} we find that $p \mid a$ 
and so $a=0.$  Hence $u^p = \pm 1$ and we reach the contradiction that $u = \pm 1 \in U.$\\
\\
Hence $p=2$ and $u^2=\pm \rho^a\mu^b .$  Since $\pm u^2\mu^{-b}$ is either totally positive or totally 
negative and since $\rho$ is neither totally positive nor totally negative, $a=0.$  Furthermore, since both $\mu$ and $u^2$ are totally positive, we are forced to conclude that $u^2=\mu^b.$  If $b=0,$ then $u=\pm 1 \in U;$ contradiction.
Therefore $b=1$ and $u^2=\mu.$  Since $\mu \equiv 27/4 \pmod{\mathcal{L}_2},$ we conclude 
that $(3\,\vert\,\ell_2)_2 = 1$ and so contradict condition \ref{ng_lem7c}.  Having 
exhausted the possibilities, we are forced to conclude that 
$p \nmid \lbrack \o_K^\times : U \rbrack.$
\end{proof}

\begin{proposition}\label{ng_prop3}
Let $y \in \Z$ with $(37\times 47,y)=1$ and such that $y^r = 6n+7.$  If for each 
prime $p \mid r$ there are primes $\ell_i \mid (6n+7)$ $(i=1,2,3)$ satisfying 
\begin{enumerate}[label=(\roman*)]
\item $(2\,\vert\,\ell_1)_p=(37\,\vert\,\ell_1)_p=1,$ and $(3\,\vert\,\ell_1)_{p}\neq 1,$\label{ng_prop3a}
\item $(37\,\vert\,\ell_2)_p=1,$ $(2\,\vert\,\ell_2)_p \neq 1,$ and $(3\,\vert\,\ell_2)_{p}\neq 1,$ \label{ng_prop3b}
\item $(37\,\vert\,\ell_3)_p\neq 1,$\label{ng_prop3c}
\end{enumerate}
then $(2-\rho) = J^r$ where $\lbrack J \rbrack$ has order $r.$
\end{proposition}

\begin{proof}
Since the hypotheses of Proposition \ref{ng_prop2} hold, we have that 
$(2-\rho) = J^r$ for some ideal $J \in \mathcal{I}_K$ and so the order of $\lbrack J \rbrack$ is 
$\le r.$  If it is strictly $< r,$ then $(2-\rho)$ is a prime power of a principal ideal for 
some prime $p \mid r.$  Supposing this to be the case and writing 
$(2-\rho)=(\beta)^p,$ for some $\beta \in \o_K,$ we have that 
$2-\rho = u\beta^p$ for some $u \in \o_K^\times.$  Since the hypotheses of Lemma \ref{ng_lem7} hold 
we have that $p \nmid \lbrack \o_K^\times : U \rbrack.$  
Writing $1=gp+h\lbrack\o_K^\times:U\rbrack$ for some $g,h \in \Z,$ we find that 
$$
2-\rho = u\beta^p = u^{gp+h\lbrack\o_K^\times:U\rbrack}\beta^p 
=(\pm \rho^{a^\prime}\mu^{b^\prime})^h(u^g\beta)^p 
=\pm \rho^a\mu^b(\beta^\prime)^p.
$$
For notational convenience, replace $\beta^\prime$ with $\beta$ and note that 
the same relation holds in the other two isomorphic cubic number fields to conclude that
$2-\rho_i = \pm \rho_i^a\mu_i^b\beta_i^p$ $(i=0,1,2).$\\
\\
Choose prime ideals $\mathcal{L}_i \subset \o_K,$ $(i=1,2,3),$ with $\mathcal{L}_i \cap \Z = \ell_i \Z$ 
and such that $\rho \equiv 2 \pmod{\mathcal{L}_i}.$  Then $\mu \equiv 27/4 \pmod{\mathcal{L}_i}.$
On the one hand, 
$(2-\rho_1)(2-\rho_2) = (\rho_1\rho_2)^a(\mu_1\mu_2)^b(\beta_1\beta_2)^p = 
\rho^{-a}\mu^{-b}(\beta_1\beta_2)^p \equiv 2^{2b-a}3^{-3b}(\beta_1\beta_2)^p \pmod{\mathcal{L}_i}.$
On the other hand, 
$(2-\rho_1)(2-\rho_2)=\rho^2+(2+n)\rho+(4+3n)\equiv 5n+12 \pmod{\mathcal{L}_i} \equiv 37/6 
\pmod{\mathcal{L}_i}.$
Hence $37 \equiv 2^{2b-a+1}3^{1-3b}(\beta_1\beta_2)^p \pmod{\mathcal{L}_i}.$  
Lemma \ref{ng_lem6} guarantees that $\ell_i \nmid \disc(f_n)$ and so 
Lemma \ref{ng_lem5} then guarantees that $\lbrack \o_K/\mathcal{L}_i : \Z/\ell_i\Z \rbrack = 1.$  
Thus $(2^{a-2b-1}3^{3b-1}37\,\vert\,\ell_i)_p = 1.$  Applying condition \ref{ng_prop3a}, we conclude that 
$3b -1 \equiv 0 \pmod{p}.$  Using this result along with condition \ref{ng_prop3b}, we find that 
$a-2b-1 \equiv 0 \pmod{p}.$  Since $p$ divides $3b-1$ and $a-2b-1,$ we deduce that 
$(37\,\vert\,\ell_3)_p = 1$ and so contradict condition \ref{ng_prop3c}.  We conclude that 
$(2-\rho)$ is not a prime power of a principal ideal for any prime $p \mid r$ and 
hence that $\lbrack J \rbrack$ has order $r.$
\end{proof}

\noindent Having established sufficient conditions for a field $K \in \frak{F}$ to posses an ideal class of arbitrary nontrivial order $r \not\equiv 0 \pmod{3},$ we now prove Proposition \ref{ng_prop1} by showing that for each such $r$ 
there is a field $K \in \frak{F}$ satisfying these conditions.

\begin{proof}[Proof of Proposition \ref{ng_prop1}]
Let $p$ be a prime dividing $r.$  Let $L_1 = \Q(\zeta_p,2^{1/p},37^{1/p})$ and 
$L_2= \Q(\zeta_{2p},3^{1/p}),$ where $\zeta_n$ denotes a primitive $n$th root of unity.  A prime 
$\ell$ totally splits in $L_1$ if and only if $\ell \equiv 1 \pmod{p}$ and 
$(2\,\vert\,\ell)_p=(37\,\vert\,\ell)_p=1.$  Similarly a prime $\ell$ totally splits in $L_2$
if and only if $\ell \equiv 1 \pmod{2p}$ and $(3\,\vert\,\ell)_p=1.$  Since 
$L_2 \not\subset L_1,$ we apply Bauer's Theorem to conclude that there are infinitely 
many primes $\ell$ that totally split in $L_1$ but not in $L_2.$  Hence 
there are infinitely many primes $\ell$ for which 
$(2\,\vert\,\ell)_p=(37\,\vert\,\ell)_p=1,$ and $(3\,\vert\,\ell)_{p}\neq 1.$  In other words, 
there are infinitely many primes $\ell$ satisfying condition \ref{ng_prop3a} of Proposition \ref{ng_prop3}.  
Analogous arguments show that there are infinitely many primes satisfying 
condition \ref{ng_prop3b} and infinitely many primes satisfying condition \ref{ng_prop3c} of Proposition \ref{ng_prop3}.\\
\\
For each prime $p \mid r,$ choose a triple of primes $\{\ell_{i,p}\}_{i=1}^3$ satisfying conditions 
\ref{ng_prop3a}-\ref{ng_prop3c} of Proposition \ref{ng_prop3} making sure never to choose $2,3,37,$ or 
$47.$  Let $s \in \N$ with $(s,37\times 47)=1$ and such that 
$s \left(\prod\limits_{\substack{\textmd{primes} \\ p \mid r}} \prod_{i=1}^3 \ell_{i,p}\right) \equiv 1 \pmod{6}.$  Set $y:= s \left(\prod\limits_{\substack{\textmd{primes} \\ p \mid r}} \prod_{i=1}^3 \ell_{i,p}\right).$
If $y^r \equiv 3 \pmod{4},$ then replace $y$ with $y^2$ so that $y^r \equiv 1 \pmod{4}.$  
Since $y^r \equiv 7 \pmod{6},$ there is an $n \in \N$ such that 
$y^r = 6n+7.$  Reducing this equation $\pmod{4}$ reveals that $n$ is odd.  
Furthermore, since $y\ge s\left(\prod\limits_{\substack{\textmd{primes} \\ p \mid r}} \prod_{i=1}^3 \ell_{i,p}\right)$ 
and since $s \in \N$ and $\ell_i \neq 2,3$ we deduce that $y \ge 125$ and 
hence that $n\ge 5.$  We see that the hypotheses of Proposition \ref{ng_prop3} are satisfied and hence 
that there is an ideal $J \in \mathcal{I}_{K_n}$ whose class $\lbrack J \rbrack$ has order $r.$
\end{proof}

\normalsize
\baselineskip=17pt

%\begin{comment}
\providecommand{\bysame}{\leavevmode\hbox to3em{\hrulefill}\thinspace}
\providecommand{\MR}{\relax\ifhmode\unskip\space\fi MR }
% \MRhref is called by the amsart/book/proc definition of \MR.
\providecommand{\MRhref}[2]{%
  \href{http://www.ams.org/mathscinet-getitem?mr=#1}{#2}
}
\providecommand{\href}[2]{#2}

%\end{comment}


\begin{thebibliography}{1}

\bibitem{gras3}
M.-N. Gras, \emph{M\'ethodes et algorithmes pour le calcul num\'erique du
  nombre de classes et des unit\'es des extensions cubiques cycliques de
  $\mathbb{Q}$}, J. reine angew. Math. \textbf{277} (1975), 89--116.

\bibitem{gras2}
\bysame, \emph{Table num\'erique du nombre de classes et des unit\'es des
  extensions cycliques r\'eelles de degr\'e $4$ de $\mathbb{Q}$}, Publ. Math.
  Besan\c{c}on (1977/78), 1--79.

\bibitem{gras1}
\bysame, \emph{Families d'unit\'es dans les extensions cycliques r\'eelles de
  degr\'e 6 de $\mathbb{Q}$}, Publ. Math. Besan\c{c}on \textbf{2}
  (1984/85-1985/86), 1--26.

\bibitem{hilbert}
D.~Hilbert, \emph{Ueber die {I}rreducibilit\"at ganzer rationaler {F}unctionen
  mit ganzzahligen {C}oefficienten}, J. reine angew. Math. \textbf{110} (1892),
  104--129.

\bibitem{lang}
S.~Lang, \emph{Algebra}, 3rd ed., Graduate Texts in Mathematics, vol. 211,
  Springer, 2002.

\bibitem{lazarus}
A.J. Lazarus, \emph{The class number and cyclotomy of simplest quartic
  fields}, Ph.D. thesis, Berkeley, 1989.

\bibitem{nagell}
T.~Nagell, \emph{\"{U}ber die {K}lassenzahl imagin\"ar-quadratischer
  {Z}ahlk\"orper}, Abh. Math. Seminar Univ. Hamburg \textbf{1} (1922),
  140--150.

\bibitem{pari}
The PARI~Group, PARI/GP version \texttt{2.11.4}, Univ. Bordeaux, 2022, \url{http://pari.math.u-bordeaux.fr/}.

\end{thebibliography}
\end{document}